\documentclass[11pt, reqno]{amsart}

\usepackage{amssymb,latexsym,amsmath,amsfonts,amsthm}
\usepackage{mathrsfs}
\usepackage{enumitem}
\usepackage[usenames]{color}
\usepackage{hyperref}
\usepackage{comment}

\allowdisplaybreaks

\numberwithin{equation}{section}

\theoremstyle{definition}
\newtheorem{definition}{Definition}[section]
\newtheorem{question}[definition]{Question}

\theoremstyle{remark}
\newtheorem{remark}[definition]{Remark}

\theoremstyle{plain}
\newtheorem{theorem}[definition]{Theorem}
\newtheorem{result}[definition]{Result}
\newtheorem{lemma}[definition]{Lemma}
\newtheorem{proposition}[definition]{Proposition}
\newtheorem{example}[definition]{Example}
\newtheorem{corollary}[definition]{Corollary}

\newcommand{\eps}{\varepsilon}
\newcommand{\zt}{\zeta}
\newcommand{\zbar}{\overline{z}}
\newcommand{\tht}{\theta}
\newcommand{\bas}{\boldsymbol{\epsilon}}
\newcommand{\Rem}{\mathcal{R}}

\newcommand{\kob}{{\sf k}}

\newcommand\pd[3]{\frac{\partial^{{#3}}{#1}}{\partial{#2}}}
\newcommand{\laplc}{\triangle}
\newcommand{\cHess}{\mathfrak{H}_{\raisebox{-2pt}{$\scriptstyle {\mathbb{C}}$}}}

\newcommand{\bdy}{\partial}
\newcommand{\OM}{\Omega}
\newcommand{\D}{\mathbb{D}}
\newcommand{\ome}{\mathcal{D}}

\newcommand{\U}{\mathscr{U}}

\newcommand{\smoo}{\mathcal{C}}
\newcommand{\hol}{\mathcal{O}}

\newcommand{\ahat}{\widehat{a}}
\newcommand{\bcdot}{\boldsymbol{\cdot}}

\newcommand{\lrarw}{\longrightarrow}

\newcommand{\iness}{\partial^{\raisebox{-1pt}{$\scriptstyle{{\coll}}$}}\!}
\newcommand\enve[1]{{\widehat{{#1}}}^{\raisebox{-1pt}{$\scriptstyle{{\coll}}$}}}
\newcommand\nonh[1]{{#1}^{\raisebox{0pt}{$\scriptstyle{{\mathcal{N}}}$}}}

\newcommand{\excp}{\mathcal{E}}
\newcommand{\coll}{\mathcal{S}}
\newcommand{\tgt}{\widetilde{{\ T}}}
\newcommand{\ombar}{\overline{\Omega}}
\newcommand{\inc}{\text{\textbf{\j}}}
\newcommand{\wt}{\widetilde}
\newcommand{\ov}{\overline}

\newcommand{\acorn}{\mathscr{Q}}
\newcommand{\bloz}{{\text{\scriptsize{$\blacklozenge$}}}}
\newcommand{\hkim}{\text{H.~\!Kim~\emph{et al.}}}

\newcommand{\Z}{\mathbb{Z}}
\newcommand{\N}{\mathbb{N}}

\newcommand{\Cn}{\mathbb{C}^n}

\newcommand{\C}{\mathbb{C}} 
\newcommand{\R}{\mathbb{R}}

\newcommand{\re}{{\sf Re}}
\newcommand{\im}{{\sf Im}}

\begin{document}

\title[Holomorphic homogeneous regular domains \& the squeezing function]{A new family of holomorphic homogeneous
regular \\ domains and some questions on the squeezing function}

\author{Gautam Bharali}
\address{Department of Mathematics, Indian Institute of Science, Bangalore 560012, India}
\email{bharali@iisc.ac.in}

\begin{abstract}
We revisit the phenomenon where, for certain domains $D$, if the squeezing function $s_D$ extends
continuously to a point $p\in \bdy{D}$ with value $1$, then $\bdy{D}$ is strongly pseudoconvex
around $p$. In $\C^2$, we present weaker conditions under which the latter conclusion is
obtained. In another direction, we show that there are bounded domains $D\Subset \C^n$, $n\geq 2$,
that admit large $\bdy{D}$-open subsets $\mathscr{O}\subset \bdy{D}$ such that $s_D\to 0$
approaching any point in $\mathscr{O}$. This is impossible for planar domains. We pose a few questions
related to these phenomena. But the core result of this paper identifies a new family of
holomorphic homogeneous regular domains. We show via a family of examples how abundant domains satisfying
the conditions of this result are.  
\end{abstract}


\maketitle

\vspace{-0.6cm}
\section{Introduction and statement of results}\label{S:intro}
The unifying theme of the results in this paper is a closer look at the squeezing function.
We begin by recalling its definition. The \emph{squeezing function} of a bounded domain $D\subset \C^n$, 
denoted by $s_D$, is defined as
\[
  s_D(z)\,:=\,\sup\{s_D(z;F) \mid F: D\to B^n(0,1) \; \text{is a holomorphic embedding with $F(z)=0$}\}
\]
where, for each $F: D\to B^n(0,1)$ as above, $s_D(z;F)$ is given by
\[
  s_D(z;F)\,:=\,\sup\{r>0 : B^n(0,r)\subset F(D)\}.
\]
The squeezing function was introduced by Deng--Guan--Zhang \cite{dengGuanZhang:spsfbd12}, who were
motivated by closely related notions in the work of Yeung \cite{yeung:gdusf09}. Clearly, $s_D$ is a biholomorphic
invariant. This object has been the centre of considerable attention with regard to the following
question: what complex-analytic properties of $D$ do the properties of $s_D$ encode?
This paper, in some sense, focuses on the crudest properties that $s_D$ can have on reasonably
large classes of domains. 

\subsection{Properties of $\boldsymbol{\bdy{D}}$}\label{SS:boubdary} Our first result is motivated
by the following pair of theorems.

\begin{result}\label{R:h-ext}
Let $D$ be a bounded domain in $\C^n$ and let $p\in \bdy{D}$. 
\begin{enumerate}[leftmargin=27pt, label=$(\alph*)$]
  \item\label{I:Verma_Mahajan} {\rm (Mahajan--Verma, \cite[Theorem~1.2]{mahajanVerma:ctbi19})}
  Suppose $\bdy{D}$ is $\smoo^\infty$-smooth, Levi-pseudoconvex and of finite type near $p$, and $p$ is
  an $h$-extendible boundary point. If $\lim_{D\ni z\to p}s_D(z) = 1$, then $p$ is a strongly
  pseudoconvex point.
  \item\label{I:Nikolov} {\rm (Nikolov, \cite[Theorem~1]{nikolov:bsfnhebp18})} Suppose $D$ is pseudoconvex
  and has $\smoo^\infty$-smooth boundary, and $p$ is an $h$-extendible boundary point. If,
  for a nontangential sequence $\{z_\nu\}\subset D$ with $z_\nu\to p$, $\lim_{\nu\to \infty}s_D(z_{\nu}) = 1$,
  then $p$ is a strongly pseudoconvex point.
\end{enumerate}
\end{result}

Since $h$-extendible boundary points will not be a focus of this paper, we shall not formally define what they are
(nor what the Catlin multitype\,---\,which is required to define $h$-extendibility\,---\,is). But to cite an example
of $h$-extendibility: if a boundary point $p$ of a domain $D\subset \C^2$, such that $\bdy{D}$ is smooth and
pseudoconvex around it, is of finite type, then $p$ is $h$-extendible. We must also mention here the
result of Joo--Kim \cite{jooKim:bpwsft118} for domains in $\C^2$, which Result~\ref{R:h-ext} generalises.
\smallskip

A closer look at the \emph{proof} of Result~\ref{R:h-ext}-\ref{I:Verma_Mahajan} reveals
that, with the hypotheses thereof, one does \textbf{not} require $s_D$ to extend continuously to
$D\cup\{p\}$ to infer that $p$ is a strongly pseudoconvex point. The proof by 
Verma--Mahajan shows that, given all other assumptions, if, for a sequence $\{z_\nu\}$ contained
in the inward normal to $\bdy{D}$ at $p$ with $z_\nu\to p$, $\lim_{\nu\to \infty}s_D(z_{\nu}) = 1$,
then $p$ is a strongly pseudoconvex point. This resembles the hypothesis in
Result~\ref{R:h-ext}-\ref{I:Nikolov}, which, however, requires a global condition on $\bdy{D}$. These
observations would lead an analyst to ask what may be
the weakest conditions on the triple $(p, \bdy{D}, s_D)$ that would lead to the conclusion in
Result~\ref{R:h-ext}. Specifically, given a bounded domain $D$ and $p\in \bdy{D}$, assuming:
\begin{itemize}[leftmargin=24pt]
  \item existence of a sequence $\{z_{\nu}\}\subset D$ with $z_{\nu}\to p$ that is
  \textbf{not} necessarily non-tangential;
  \item smoothness and Levi-pseudoconvexity of $\bdy{D}$ \textbf{only} near $p$; and
\end{itemize}
assuming $p$ is an $h$-extendible point, then does $\lim_{\nu\to \infty}s_D(z_{\nu}) = 1$ imply that $p$ is
a strongly pseudoconvex point?
\smallskip

Now, the proofs of parts~\ref{I:Verma_Mahajan} and~\ref{I:Nikolov} of Result~\ref{R:h-ext} rely, to
differing extents, on the scaling method. This method has proven to be very effective in addressing questions
involving the squeezing function. The principal difficulty in answering the above question is as follows.
Given a pair $(D,p)$ as above and a sequence
$\{z_{\nu}\}\subset D$ ($z_{\nu}\to p$) that does not approach $p$ non-tangentially, the geometry of the
limits\,---\,in the Hausdorff sense\,---\,of convergent subsequences of the sequence $\{D_{\nu}\}$ of rescaled
domains associated with $\{z_{\nu}\}$ is not clear enough to support further analysis. That said,
given that $p\in \bdy{D}$ is assumed to be $h$-extendible, one knows that there are \textbf{certain} 
partially tangential approach-regions $\Gamma\subset D$, with vertex at $p$, such that
if $\{z_{\nu}\}\subset \Gamma$, then the
difficulty alluded to does not arise.
\smallskip

Such an approach-region $\Gamma$ is determined by the Catlin multitype of $p$.
Thus, ``extending'' Result~\ref{R:h-ext}-\ref{I:Nikolov} by assuming $\{z_{\nu}\}$ to lie in $\Gamma$
is a bit misconceived: if one knows what $\Gamma$ is, then one already knows
the Catlin multitype of $p$ and, hence, whether or not $p$ is a strongly pseudoconvex point! Instead, since we
assume that $\bdy{D}$ is smooth around $p$, the natural approach regions
to consider (\textbf{but} see Remark~\ref{R:BS} below)\,---\,which are genuinely
osculatory (i.e., \textbf{not} non-tangential)\,---\,are those implicit in our first result:
Theorem~\ref{T:conv_to_1} below.
But first, we give a definition. 

\begin{definition}\label{D:parabol}
Let $D\varsubsetneq \C^n$ be a domain, $p\in \bdy{D}$, and assume that $\bdy{D}$ is $\smoo^1$-smooth
near $p$. Let $\pi_p$ denote the (real) orthogonal projection of $\R^{2n}$\,$(\cong \C^{n})$ onto the
tangent space $T_p(\bdy{D})$ relative to the standard inner product on $\R^{2n}$. We say that a sequence
$\{z_\nu\}\subset D$ with $z_\nu\to p$ has \emph{paraboloidal approach to $p$} if there exists a $C>0$ such that
\begin{equation}\label{E:parabol}
  \|\pi_p(z_{\nu}-p)\|\,\leq\,C\sqrt{{\rm dist}(z_{\nu}; \bdy{D})} \; \; \text{$\forall \nu$ sufficiently large.}
\end{equation}
\end{definition}

We now provide an answer to the question raised above. For simplicity, we state our first result for
domains in $\C^2$. We defer discussions on what can be said for domains in $\C^n$, $n\geq 3$, to
Section~\ref{S:ques}

\begin{theorem}\label{T:conv_to_1}
Let $D$ be a bounded domain in $\C^2$ and let $p\in \bdy{D}$. Suppose $\bdy{D}$ is $\smoo^\infty$-smooth and
Levi-pseudoconvex near $p$, and $p$ is of finite type. If, for some sequence $\{z_\nu\}\subset D$
with paraboloidal approach to $p$, $\lim_{\nu\to \infty}s_D(z_{\nu}) = 1$,
then $p$ is a strongly pseudoconvex point.
\end{theorem}

\begin{remark}\label{R:BS}
In a recent paper by H.~\!Kim~\emph{et al.} in \emph{Bull. Korean Math. Soc.} {\bf 57} (2020), the authors
assert the conclusion of Theorem~\ref{T:conv_to_1} for a class of (pointed) pseudoconvex domains in $\Cn$
that, when $n=2$, is precisely the subclass of pointed domains $(D, p)$ considered in
Theorem~\ref{T:conv_to_1} that are globally pseudoconvex\,---\,and with no restriction on how
$\{z_{\nu}\}$ approaches $p$. There is a notable gap in their proof. We shall elaborate
upon  this in Remark~\ref{R:BS_discuss},
once we have introduced the requisite notation. The nature of the gap in the aforementioned
proof provides yet another motivation for Theorem~\ref{T:conv_to_1} and for the notion introduced
by Definition~\ref{D:parabol}.
\end{remark}

The existence of just one sequence $\{z_\nu\}$ with the properties stated in the theorem above is an
example of the ``crude properties'' of $s_D$ alluded to earlier in this section. Now, the specific way in which we
manage the difficulty discussed above\,--- \,i.e., concerning the sequence $\{D_{\nu}\}$ of rescaled
domains\,--- \,holds out the possibility of extending
Theorem~\ref{T:conv_to_1} to certain well-studied subclasses of domains with (locally) $h$-extendible
boundaries. The reader is directed to Question~\ref{Q:corank_1} and Question~\ref{Q:convex}
below, and to the discussions that precede them.

\subsection{A family of holomorphic homogeneous regular domains}\label{SS:HHR}
A lot of the recent literature on the squeezing function has focused on the phenomenon discussed in the
last subsection. There has, in a sense, been a reduced focus on the property that gives the
squeezing function much of its salience. To discuss this property, we need a definition.

\begin{definition}\label{D:HHR}
A bounded domain $D\subset \C^n$ is said to be \emph{holomorphic homogeneous regular} if
$\inf_{z\in D}s_D(z) > 0$.
\end{definition}

The term ``holomorphic homogeneous regular'' has its origins in \cite{liuSunYau:cmmsRsI04} by
Liu--Sun--Yau. The mere fact that $\inf_{z\in D}s_D(z) > 0$ endows $D$ with many pleasant properties that
make its complex geometry extremely well-behaved. For instance, the classical intrinsic metrics on $D$ are
mutually bi-Lipschitz equivalent. This is one of the many results in \cite{yeung:gdusf09} (which, in turn,
is motivated by the work \cite{liuSunYau:cmmsRsI04} in which the bi-Lipschitz equivalece of the Bergman,
Carath{\'e}odory and Kobayashi metrics on the Teichm{\"u}ller spaces is demonstrated).    
It would thus
be of interest to add to the list of domains that are holomorphic homogeneous regular (see 
\cite[Section~1]{kimZhang:uspbccdCn16} and its bibliography for the classes of domains that are known to have
this property). In Theorem~\ref{T:HHR}, we introduce a new family of domains that are holomorphic
homogeneous regular. But we first need a definition. (The word ``ball'' will
mean a ball with respect to the Euclidean norm on $\C^n$.)

\begin{definition}\label{D:well_cnvx}
Let $D\varsubsetneq \C^n$ be a domain and let $p\in \bdy{D}$. We say that $D$ is \emph{well-convexifiable near
$p$} if a $\bdy{D}$-open neighbourhood of $p$ is $\smoo^1$-smooth and there exist:
\begin{itemize}[leftmargin=24pt]
  \item a biholomorphic map $\Psi: D\to \C^n$ that extends continuously to $D\cup \{p\}$; and
  \item an open ball $B$ with centre $\Psi(p)$;
\end{itemize}
such that $\Psi(D)\cap B$ is convex,
$\bdy\Psi(D)\cap B$ is a $\smoo^1$-smooth strictly convex hypersurface (i.e., for each
$q\in \bdy\Psi(D)\cap B$, its tangent hyperplane at $q$ intersects $(\overline{\Psi(D)}\cap B)$ only at $q$),  
and the tangent hyperplane $H$ of $\bdy\Psi(D)$ at $\Psi(p)$ satisfies $H\cap \overline{\Psi(D)} = \{\Psi(p)\}$.
\end{definition}

\begin{theorem}\label{T:HHR}
Let $D$ be a bounded domain in $\C^n$ having $\smoo^1$-smooth boundary. Assume that for each
boundary point $p$ such that $\bdy{D}$ is \textbf{not} strongly Levi-pseudoconvex at $p$, 
$D$ is well-convexifiable near $p$. Then, $D$ is holomorphic homogeneous regular.
\end{theorem}

\begin{remark}\label{R:strong_Levi}
For $D\Subset \C^n$ as above, if $\bdy{D}$ \textbf{is} strongly Levi-pseudoconvex
at a point $\xi\in \bdy{D}$, it is understood that for some $\bdy{D}$-open neighbourhood $\omega$
of $\xi$, $\omega$ is a strongly pseudoconvex hypersurface. However, it suffices in
Theorem~\ref{T:HHR} for $\bdy{D}$ to have $\smoo^1$ regularity globally.
\end{remark}

The notion of $D$ being well-convexifiable near a point $p\in \bdy{D}$ might be reminiscent of
the notion of a ``g.s.c. boundary point'' introduced by Deng \emph{et al.} in \cite{dengGuanZhang:psfgtbd16}.
However, we note certain significant points of contrast:
\begin{itemize}[leftmargin=24pt]
  \item Definition~\ref{D:well_cnvx} requires $\Psi(D)$ to be convex merely \textbf{close to} $\Psi(p)$.
  \item Observe the low regularity of $\bdy{D}$ in Definition~\ref{D:well_cnvx}.
  \item In case $\bdy\Psi(D)$ is $\smoo^2$-smooth around $\Psi(p)$, Definition~\ref{D:well_cnvx}
  does \textbf{not} require $\bdy\Psi(D)$ to be strongly convex at $\Psi(p)$.
\end{itemize}
The reader may ask: how restrictive are the conditions of Theorem~\ref{T:HHR}? The discussion of
Example~\ref{Ex:HHR} below will demonstrate how abundant such domains (i.e., those
satisfying the conditions of Theorem~\ref{T:HHR}) are, even in $\C^2$. Also, highlighting
the third bullet-point above: the domains discussed in Example~\ref{Ex:HHR} can even have
boundary points of infinite type.

\subsection{The decay of $\boldsymbol{s_D}$ approaching $\boldsymbol{\bdy{D}}$}\label{SS:decay}
We begin this subsection with an estimate on $s_D$ for a certain class of domains, and broaden our
analysis from there. In what follows, $K_X$ will denote the Kobayashi pseudodistance on a
connected complex manifold $X$.

\begin{theorem}\label{T:squee_0_I}
Let $\OM$ be a bounded domain in $\C^n$, $n\geq 2$. Let $K$ be a compact subset of $\OM$
such that $\OM\setminus K$ is connected. Write $D := \OM\setminus K$. Then
\begin{equation}\label{E:squee_0}
  s_D(z)\,\leq\,\tanh\big(K_{\OM}(z; \bdy{D}\cap K)\big) \quad \forall z\in D.
\end{equation}
\end{theorem}

It follows from \eqref{E:squee_0} that the domain $D$ above is not holomorphic
homogeneous regular. This is not surprising since, by the Hartogs phenomenon, $D$ is not pseudoconvex
(which is a necessary condition for holomorphic homogeneous regularity), but \eqref{E:squee_0}
gives us a quantitative estimate for $s_D$. Note also that when $K$ is such that $\bdy{K}$ has finitely
many connected components, and $K$ is so large that no connected component of
$\bdy{K}$ is a single point, then\,---\,in view of
\cite[Theorem~5.2]{dengGuanZhang:spsfbd12}\,---\,Theorem~\ref{T:squee_0_I} describes a
phenomenon that is \emph{impossible for $(\OM\setminus K)\subset \C$ with $K$ having the latter properties}. The proof of
Theorem~\ref{T:squee_0_I}, as we shall see, relies crucially on $D$ \textbf{not} being pseudoconvex.
This, together with the last observation, raises the following

\begin{question}\label{Q:squee_0}
Let $D$ be a bounded domain in $\C^n$, $n\geq 2$, such that for a set $\omega\subset \bdy{D}$
that is large in some appropriate sense $\lim_{D\ni z\to p}s_{D}(z) = 0$
for each $p\in \omega$. Then, does it follow that $D$ is not pseudoconvex?
\end{question}

The answer to this question will follow from a generalisation of Theorem~\ref{T:squee_0_I}. We must give a definition
before we can state such a result.

\begin{definition}
Let $D\varsubsetneq \C^n$, $n\geq 2$, be a domain and let $\coll$ be a non-empty subset of $\hol(D)$.
The \emph{$\coll$-inessential boundary of $D$}, denoted by $\iness{D}$, is defined as
\begin{align*}
  \iness{D}:=\big\{\xi\in \bdy{D} : \  &\exists U, \ \text{a connected open neighbourhood
  of $\xi$, and $V$, a connected component} \\
  &\text{of $D\cap U$ s.t., for each $f\in  \coll$, $\exists F_f\in \hol(U)$ satisfying  
  $\left. f\right|_{V} = \left. F_f\right|_{V}$}\big\}.
\end{align*}
\end{definition}
Now, recall that for any non-empty open set $D\subset \C^n$, $H^\infty(D)$ denotes the set of bounded
holomorphic functions on $D$. With these definitions, we can state our next result.

\begin{theorem}\label{T:squee_0_II}
Let $D$ be a bounded domain in $\C^n$, $n\geq 2$, and let $\coll$ be a subset of $\hol(D)$
that contains $H^\infty(D)$. Suppose $\iness{D}\neq \varnothing$. Then
\[
  \lim_{D\ni z\to p}s_D(z)\,=\,0 \; \; \text{for each $p\in \iness{D}$}.
\]
\end{theorem}

The corollary below provides an answer (in the negative) to Question~\ref{Q:squee_0}.
We must specify a sense in which a subset of $\bdy{D}$\,---\,$D$ as in
Question~\ref{Q:squee_0}\,---\,is large. To this end: recall that if $A$ is a nowhere dense subset of $\C^n$ and
$q\in A$, then we say that $A$ is \emph{locally separating at $q$} if there exists an open ball $B$ with centre $q$
such that $B\setminus A$ has more than one connected component.  

\begin{corollary}\label{C:answer}
For each $n\in \Z_+\setminus\{1\}$, there exists a bounded pseudoconvex domain $D_n\subset \C^n$ that
admits a non-empty $\bdy{D_n}$-open set $\omega_n\subset \bdy{D_n}$ that is locally separating at 
some point in it and such that
$\lim_{D\ni z\to p}s_{D_n}(z) = 0$ for each $p\in \omega_n$.
\end{corollary}

Example~\ref{Ex:squee_0} below will serve as the proof of Corollary~\ref{C:answer}. Note that 
the condition that $\omega_n$ must be locally separating at some point in it
rules out, for instance, $\omega_n$ being a complex subvariety of some domain
$\OM$ containing $D_n$. Thus, in a certain sense, $\omega_n$ is large.  
\medskip

\section{Examples}\label{S:examples}

We now present the examples mentioned in Section~\ref{S:intro}. Certain notations
(which have also been used without clarification in Section~\ref{S:intro}) recur
in their presentation. We first explain these notations.

\subsection{Common notations}\label{SS:notation} We fix the following notation used in this section (which
will recur elsewhere in this paper):
\begin{enumerate}[leftmargin=27pt]
  \item $\|\bcdot\|$ will denote the Euclidean norm on $\C^n$ or $\R^N$. Unless specifically
  mentioned, distances will be understood to mean the Euclidean distance. Thus, for two points
  $a$ and $b$ in Euclidean space, we shall write ${\rm dist}(a, b)$ as $\|a-b\|$.
  \vspace{0.65mm}
  
  \item The open ball in $\C^n$ with centre $a$ and radius $r$ will be denoted by $B^n(a, r)$. However,
  we shall write $B^1(a, r)$ (i.e., $a\in \C$) as $D(a, r)$.
\end{enumerate}
\smallskip

\subsection{The two examples} 
To present our first example, we shall need the simple result:

\begin{lemma}\label{L:radial_subh}
Let $\varphi: [0, R)\to \R$, $0<R\leq +\infty$, be a function of class $\smoo^2\big([0, R)\big)$,
and  suppose $\varphi^{(n)}(x) = o(x^{2-n})$ as $x\to 0^+$ for $n = 1, 2$. Define the function
$\phi : D(0, R)\to \R$ by $\phi(z) := \varphi(|z|)$ for each $z\in D(0, R)$. 
If $\varphi$ satisfies the inequality
\[
  \varphi^{\prime\prime}(x) + x^{-1}\varphi^\prime(x)\,\geq\,0 \; \; \forall x\in (0, R),
\]
then $\phi$ is subharmonic.
\end{lemma}

The proof of the above lemma is an elementary calculation. It follows from the fact that, writing
$\C\setminus \{0\}\ni z = re^{i\tht}$, the Laplacian in polar coordinates has the expression:
\[
  \laplc\,:=\,\pd{{}}{r^2}{2} + \frac{1}{r}\pd{{}}{r}{{}}
  		+ \frac{1}{r^2}\pd{{}}{\tht^2}{2}.
\]

\begin{example}\label{Ex:HHR}
An example of a family of domains $D\Subset \C^2$ that are not convex and such that:
\begin{itemize}[leftmargin=24pt]
  \item each $D$ has boundary points around which $\bdy{D}$ is not strongly Levi-pseudoconvex; and
  \item each $D$ has all the properties stated in Theorem~\ref{T:HHR}.
\end{itemize}
\end{example}

\noindent{Fix a small number $\eps>0$ and let $\sigma: [0, 1+\eps)\to \R$ be a function of
class $\smoo^\infty\big([0, 1+\eps)\big)$ having the following properties:
\begin{enumerate}[leftmargin=27pt]
  \item\label{I:decay} $\sigma(0) = 0$, $\sigma^\prime(x) = o(x)$, and $\sigma^{\prime\prime}(x) = o(x)$ 
  as $x\to 0^+$.
  \item\label{I:one} $\sigma(1) = 1$.
  \item\label{I:incr} $\sigma^\prime(x) > 0$ for every $x\in (0, 1+\eps)$
  \item $\sigma$ satisfies the inequality
  \begin{equation}\label{E:strictly_subh}
    x\sigma^{\prime\prime}(x) + \sigma^\prime(x)\,>\,0 \; \; \forall x\in (0, 1+\eps).
  \end{equation}
  \item\label{I:cvx_ncvx} $\sigma^{\prime\prime}(x) > 0$ for each $x\in (0, \eps)$ but the
  function $[0, \sqrt{1+\eps})\ni x\longmapsto \sigma(x^2)$ is \textbf{not} convex on the entire
  interval $[0, \sqrt{1+\eps})$.
\end{enumerate}
Now, let $P\in \C[Z]$ be a non-constant polynomial such that the domain
\[
 D_{P,\,\sigma}\,:=\,\{(Z,W)\in \C^2 : |W+P(Z)|^2 + \sigma(|Z|^2) < 1\}
\]
is non-convex. We will show that each domain in the family
\[
 \{D_{P,\,\sigma} : \text{$\sigma$ and $P$ satisfy all the conditions above}\}
\]
is non-convex, has all the properties stated in Theorem~\ref{T:HHR}, and
each such domain has boundary points around which its boundary is not strongly Levi-pseudoconvex.
To this end, let us \textbf{fix} a $\sigma$ and $P$ with the properties stated above. For simplicity
of notation, let us now refer to this $D_{P,\,\sigma}$ as $D$. Define
\[
  \rho(Z,W)\,:=\,|W+P(Z)|^2 + \sigma(|Z|^2) - 1 \; \; \forall (Z,W) \in D(0,\sqrt{1+\eps})\times \C.
\]}
\smallskip

That $D$ is non-convex follows from its definition (i.e., by the choice of $P$). We compute:
\[
  \partial\rho(Z,W)\,=\,\big(P^\prime(Z)(\overline{W}+\overline{P(Z)})+\overline{Z}\sigma^\prime(|Z|^2)\big)\,dZ
  + (\overline{W}+\overline{P(Z)})\,dW.
\]
Observe that by the conditions on $\sigma$, $D\subset D(0,1)\times\C\subset {\sf dom}(\rho)$.
Now, if $(Z,W)\in \bdy{D}$ and $W+P(Z)=0$, then by the conditions \eqref{I:one} and \eqref{I:incr},
necessarily $|Z|=1$. From this,
we get that $\partial\rho(Z,W)\neq 0$ for each $(Z,W)\in \bdy{D}$. These facts imply that $\rho$ is a defining
function of $D$ and that $\bdy{D}$ is $\smoo^2$-smooth. 
\smallskip

Now consider the biholomorphism $\Psi : \C^2\to \C^2$ given by
\[
  \Psi(Z,W)\,:=\,(Z, W+P(Z))\,=:\,(z,w).
\]
If we write $\OM := \Psi(D)$, then, clearly, $\varrho$ defined by
\[
  \varrho(z,w)\,:=\,|w|^2 + \sigma(|z|^2) - 1 \; \; \forall (z,w) \in D(0,\sqrt{1+\eps})\times \C,
\]
is a defining function of $\OM$. We will use this fact and the conclusions of the previous paragraph to
determine the set of points in $\bdy{D}$ around which $\bdy{D}$ is not strongly Levi-pseudoconvex.
To this end, we compute the complex Hessian of $\varrho$:
\begin{equation}\label{E:comp_Hess}
  \cHess(\varrho)(z,w)\,=\,\begin{bmatrix}
  						\,|z|^2\sigma^{\prime\prime}(|z|^2) + \sigma^\prime(|z|^2)	& 0\, \\
  						0													& 1\,
  					    \end{bmatrix} \quad \forall (z,w) \in D(0,\sqrt{1+\eps})\times \C.
\end{equation}
If we take $\varphi(x) := \sigma(x^2)$, $x\in [0,\sqrt{1+\eps})$, in Lemma~\ref{L:radial_subh}, then
combining the conclusion of Lemma~\ref{L:radial_subh} with \eqref{E:strictly_subh}, we conclude
that
\begin{itemize}[leftmargin=24pt]
  \item $\varrho$ is a plurisubharmonic function; and
  \item $\bdy{\OM}$ is strongly Levi-pseudoconvex around each point $(z,w)\in \bdy{\OM}$ such
  that $z\neq 0$.
\end{itemize}
Next, observe that the complex line contained in $T_{(0,1)}(\bdy{\OM})$ is $\C\times\{0\}$. Then,
by \eqref{E:comp_Hess} and the condition \eqref{I:decay},
$\bdy{\OM}$ is weakly pseudoconvex at the point $(0,1)\in \bdy{\OM}$. Finally, since the
maps $\tau_{\tht} : (z,w)\longmapsto (z, e^{i\tht}w)$, for each $\tht\in \R$, are automorphisms of
$\OM$, we conclude that $\{(0, w)\in \C^2 : |w| = 1\}$ is the set of weakly pseudoconvex points
of $\bdy{\OM}$. From this, we conclude:
\begin{align}
  &\text{the set $\excp:=\Psi^{-1}\big(\{(0, w)\in \C^2 : |w| = 1\}\big)$ is the set of weakly pseudoconvex
  points} \notag \\
  &\text{in $\bdy{D}$, and $\bdy{D}$ is strongly Levi-pseudoconvex around each point in $\bdy{D}\setminus\excp$.}
  \label{E:bdy_geom}
\end{align}

In view of \eqref{E:bdy_geom}, it suffices to show that $D$ is well-convexifiable near each $p\in \excp$ to complete the
discussion of Example~\ref{Ex:HHR}. For each $p\in \excp$, we will see that the $\Psi$ introduced above
will play the role of the biholomorphism $\Psi$ appearing in Definition~\ref{D:well_cnvx}. To do so,
we shall show that $D$ is well-convexifiable near the point $p^* := \big(0, -(1+P(0))\big)\in \bdy{D}$. If we show this,
then (note that $\Psi(p^*) = (0, -1)$), since the unitaries $\tau_{\tht}$ are \textbf{linear} automorphisms
of $\OM$, we would be done. Let us write
\[
  z = x+iy \quad\text{and} \quad w = u+iv.
\]
We observe that $\bdy{\OM}\cap \{(z, w)\in \C^2 : u<0\}$ is given by the equation $u = g(z,v)$, where
\begin{multline*}
  g(z, v)\,:=\, -\sqrt{1-v^2 - \sigma(|z|^2)}, \\
  (z,v)\in \left\{(s+it, \eta)\in D(0,1)\times\R : |\eta| < \sqrt{1-\sigma(s^2 + t^2)}\right\}.
\end{multline*}
Thus,
\[
 g(z,v)\,=\,-1+\tfrac{1}{2}\big(v^2+\sigma(|z|^2)\big)
 + O\big(\big(v^2+\sigma(|z|^2)\big)^2\big) \; \; \text{as $\|(x,y,v)\|\to 0$}.
\]
Since $g$ is $\smoo^2$-smooth, it follows from the above and from writing out the real Hessian of $g$ 
that\,---\,in view of the first part of the condition \eqref{I:cvx_ncvx}\,---\,there is a small open ball 
$\mathfrak{B}\subset \R^3$ with centre $0$ such that $\left. g\right|_{\mathfrak{B}}$ is a convex function.
Now, the (extrinsic) tangent hyperplane of $\bdy{\OM}$ at $(0,-1)$, call it $H$, is
\[
  H\,=\,(0,-1)+\{(x+iy, u+iv)\in \C^2 : u = 0\}.
\]
By the expression for $g$ and the condition \eqref{I:incr}, it is immediate that
$H\cap\overline{\OM} = \{(0, -1)\}$. From the last three statements, we see that   
$D$ is well-convexifiable near the point $p^*$. In view of the remarks at the beginning of
this paragraph, we conclude that $D$ has all the properties stated in Theorem~\ref{T:HHR},
which concludes the discussion of Example~\ref{Ex:HHR}.\hfill $\blacktriangleleft$
\medskip

In support of certain remarks made in Subsection~\ref{SS:HHR}, we also note that:
\begin{itemize}[leftmargin=24pt]
  \item In view of the condition \eqref{I:cvx_ncvx} above, if a domain $D$ satisfies all
  the properties stated in Theorem~\ref{T:HHR}, $p$ is a boundary point
  such that $\bdy{D}$ is \textbf{not} strongly Levi-pseudoconvex around $p$,
  and $\Psi$ is the biholomorphism associated with $p$ appearing in Definition~\ref{D:well_cnvx},
  then the conditions on $D$ do \textbf{not} in general imply that $\Psi(D)$ must be
  globally convex.
  \item The conditions \eqref{I:decay}--\eqref{I:cvx_ncvx} above on $\sigma$ are such that they admit examples of
  $\sigma$ for which the weakly pseudoconvex points of $\bdy{D_{P,\,\sigma}}$ are
  infinite-type points.
\end{itemize}

In order to present the next example, we need a preliminary result. In what follows, we
shall abbreviate the unit disc $D(0,1)$ as $\D$.

\begin{lemma}\label{L:sup-norm}
Fix $n\in \Z_+$. There exists a sequence of points $\{a_{\nu}\}\subset \D^n$ that is discrete
in $\D^n$ and such that if
$f\in H^{\infty}(D^n)$ and $|f(a_{\nu})|\leq 1$ for each $\nu\in \N$, then
$\sup_{z\in \D^n}|f(z)|\leq 1$.
\end{lemma}
\begin{proof}
Consider the following set that is discrete in $\D^n$:
\begin{multline*}
  S \\
  :=\,\bigcup_{k=2}^\infty\left\{\left(\big(1-\tfrac{1}{k}\big)\exp\Big(\frac{2\pi i m_1}{k^2}\Big),\dots,
  				\big(1-\tfrac{1}{k}\big)\exp\Big(\frac{2\pi i m_n}{k^2}\Big)\right) : 
  				(m_1,\dots, m_n)\in \{1,\dots, k^2\}^n\right\}.
\end{multline*}
Pick and \textbf{fix} some enumeration of $S$ and denote it by $\{a_{\nu}\}$. 
Pick a function $f\in H^{\infty}(\D^n)$ with the property
\begin{equation}\label{E:1_bound}
  |f(a_{\nu})|\,\leq\,1 \; \; \forall a_\nu\in S,
\end{equation}
and write $\|f\|_{\infty}:=\sup_{z\in \D^n}|f(z)|$.
Consider a $k = 2, 3, 4,\dots$. For any
point $z = (z_1,\dots, z_n)\in \overline{D\big(0, (1-\frac{1}{k})\big)^n}$, since
\[
  \bdy{D(z_1, \tau/k)}\times\dots\times\bdy{D(z_n, \tau/k)}\subset \D^n \; \; \text{for each $\tau\in (0, 1)$},
\]
Cauchy's estimates give
\begin{equation}\label{E:1st_deriv}
  \left|\pd{f}{z_j}{{}}(z)\right|\,\leq\,\frac{\|f\|_{\infty}}{k^{-1}} \quad \forall z\in
  \overline{D\big(0, (1-\tfrac{1}{k})\big)^n}
\end{equation}
and for each $j = 1,\dots, n$. Now observe: if $z\in 
\mathbb{T}^n_k := \bdy{D\big(0, (1-\tfrac{1}{k})\big)}\times\dots\times\bdy{D\big(0, (1-\tfrac{1}{k})\big)}$,
there exists an $a_{\nu}\in \mathbb{T}^n_k$\,---\,let us call it $a_{\nu}^z$\,---\,such
that $a_{\nu}^z$ can be joined to $z$ by a path that is a
concatenation of at most $n$ circular arcs of length at most $\pi\big(1-\frac{1}{k}\big)/k^2$. This gives us
the (very conservative) estimate
\[
  \|a_{\nu}^z-z\|\,\leq\,\frac{n\pi}{k^2}.
\]
From this, \eqref{E:1_bound} and \eqref{E:1st_deriv}, we get
\[
  |f(z)|\,\leq\,|f(a_{\nu}^z)| + \frac{\|f\|_{\infty}}{k^{-1}}\|a_{\nu}^z-z\|\,\leq\,1 +
  \frac{\pi n \|f\|_{\infty}}{k} \quad \forall z\in \mathbb{T}^n_k.
\]
By the Maximum Modulus Theorem, we have
\[
  |f(z)|\,\leq\,1 + \frac{\pi n \|f\|_{\infty}}{k}
  \quad \forall z\in D\big(0, (1-\tfrac{1}{k})\big)^n.
\]
Hence, by taking $k\to \infty$, the result follows.
\end{proof}   

We are now in a position to present our second example. We reiterate that the example below serves as the
proof of Corollary~\ref{C:answer}.

\begin{example}\label{Ex:squee_0}
An example of a family of bounded domains $D_n$\,---\,where, for each $n\in \Z_+\setminus\{1\}$,
$D_n\Subset \C^n$\,---\,such that
\begin{itemize}[leftmargin=24pt]
  \item each $D_n$ is a pseudoconvex domain;
  \item each $D_n$ admits a non-empty $\bdy{D_n}$-open subset $\omega_n$ such that
  $\lim_{D_n\ni z\to p}s_{D_n}(z) = 0$ for every $p\in \omega_n$;
  \item there are points in $\omega_n$ at which $\omega_n$ is locally separating.
\end{itemize}
\end{example}
  
\noindent{This example is an adaptation of a construction by Sibony \cite[Proposition~1]{sibony:pfhbmC75}
to higher dimensions. We shall therefore be brief\,---\,the reader is referred to \cite{sibony:pfhbmC75} 
for any justifications that can be taken verbatim from it. Fix an $n\in \Z_+\setminus\{1\}$. Pick
a discrete sequence $\{a_{\nu}\}\subset \D^{n-1}$ as given by Lemma~\ref{L:sup-norm}. Define
\[
  \varphi(z)\,:=\,\sum_{\nu=2}^\infty\lambda_{\nu}\log\left(\frac{\|z-a_{\nu}\|}{2\sqrt{n-1}}\right)
  \quad \forall z\in \D^{n-1},
\]
where $\{\lambda_{\nu}\}$ is a summable sequence of strictly positive numbers. As
$\log\big(\|z-a_{\nu}\|/2\sqrt{n-1}\big) \in [-\infty, 0)$ for every $z\in \D^{n-1}$ and
is bounded from below by $-\log(8\sqrt{n-1})$ on $D(0, 1/4)^{n-1}$ for
each $\nu = 2, 3, 4,\dots$, our condition on $\{\lambda_{\nu}\}$ ensures that $\varphi\not\equiv -\infty$.
The function
\[
  V(z)\,:=\,\exp\big(\varphi(z)\big) \; \; \forall z\in \D^{n-1}
\]
is a plurisubharmonic function and $0<V<1$. We define
\[
  D_n\,:=\,\big\{(z, w)\in \D^{n-1}\times\C : |w| < \exp\big(-V(z)\big)\big\}.
\]
Since $0<V<1$, $D_n\subset \D^n$, and as $\varphi\not\equiv -\infty$, $D_n$ is a proper subset
of $\D^n$. Now, $D_n$ is an example of a Hartogs domain. It is a classical fact that, as $V$ is plurisubharmonic, 
$D_n$ is pseudoconvex.}
\smallskip

Let $g\in H^{\infty}(D_n)$ and suppose $\|g\|_{\infty}\leq 1$. We consider the following series development of $g$:
\begin{equation}\label{E:series}
  g(z,w)\,=\,\sum_{k=0}^\infty\,\frac{1}{k!}\,\pd{g}{w^k}{k}(z, 0)\,w^k\,\equiv\,\sum_{k=0}^\infty h_k(z)\,w^k,
\end{equation}
where
\[
  h_k(z)\,=\,\frac{1}{2\pi i}\oint_{|\zt| = \tau\exp(-V(z))} \frac{g(z, \zt)}{\zt^{k+1}}\,d\zt \quad \forall z\in \D^{n-1},
\]
and the above expression holds true for any $\tau\in (0, 1)$. For exactly the same reasons as in
the proof of \cite[Proposition~1]{sibony:pfhbmC75}, we have
\begin{align}
  |h_k(z)|\,&\leq\,e^k \; \; \forall z\in \D^{n-1}, \notag \\
  |h_k(a_{\nu})|\,&\leq\,1 \; \; \forall k\in \N \; \text{and} \; \forall \nu\in \N. \label{E:1_bound_appl}
\end{align}
These two inequalities imply that the hypothesis of Lemma~\ref{L:sup-norm} holds true for each $h_k$.
We therefore conclude that
\begin{equation}\label{E:inference}
  |h_k(z)|\,\leq\,1 \; \; \forall z\in \D^{n-1} \; \text{and} \; \forall k\in \N.
\end{equation}
(Our argument differs from Sibony's in this last step. For $n\geq 3$, one could use the Marcinkiewicz--Zygmund
theorem where Fatou's theorem is used in \cite{sibony:pfhbmC75}. Lemma~\ref{L:sup-norm} is a
``soft'' alternative.) By \eqref{E:inference},
\begin{itemize}[leftmargin=30pt]
  \item[$(*)$] the series \eqref{E:series} converges uniformly on each compact subset of $\D^n$.
\end{itemize}
\smallskip

Taking $\coll = H^{\infty}(D_n)$ in Theorem~\ref{T:squee_0_II}, we conclude from $(*)$ that
the $\coll$-inessential boundary of $D_n$ is
\[
  \omega_n\,:=\,\big\{(z, w)\in \D^{n-1}\times\C : |w| = \exp\big(-V(z)\big)\big\}.
\]
Clearly $\omega_n$ is a $\bdy{D_n}$-open subset of $\bdy{D_n}$. It is elementary to see that, by
construction, $\omega_n$ is locally separating at every point
in $\big(\omega_n\cap \{(z, w)\in \C^{n-1}\times \C : z=0\}\big)$.
It follows from Theorem~\ref{T:squee_0_II} that 
$\lim_{D_n\ni z\to p}s_{D_n}(z) = 0$ for each $p\in \omega_n$. \hfill $\blacktriangleleft$
\medskip

\section{Analytical preliminaries}\label{S:prelim}

This section is devoted to results on the function theory and geometry of domains that will be
needed in our proofs. 
The first two results will be needed in the proof of Theorem~\ref{T:conv_to_1}. First, however,
a definition. For a real-valued polynomial $P(z, \zbar)$, the \emph{degree of $P$} is the number
\[
  \max\{i+j : i, j\in \N \; \text{and the monomial $z^i\zbar^j$ occurs in $P$ with a non-zero coefficient}\}.
\] 

\begin{result}[Oeljeklaus, \cite{oeljeklaus:agchdC293}]\label{R:oelj}
For a subharmonic, non-harmonic polynomial $P$ on $\C$, define
\[
  \OM_P\,:=\,\{(z,w)\in \C^2 : \re{w} + P(z, \zbar) < 0\}
\]
and let $\nonh{P}$ denote the sum of all the non-harmonic terms of $P$.
Given subharmonic, non-harmonic polynomials $P$ and $Q$ on $\C$, $\OM_P$ and $\OM_Q$ are
biholomorphically equivalent if and only if there exist a number $\rho>0$, constants
$a\in \C\setminus\{0\}$, $b\in \C$, and a polynomial $p\in \C[z]$ such that
\[
  P(z, \zbar)\,=\,\rho\re\big({p(z)}\big) + \rho Q(az+b, \ov{az}+\ov{b}).
\]
In particular, if $\OM_P$ and $\OM_Q$ are biholomorphically equivalent, then the degrees
of $\nonh{P}$ and $\nonh{Q}$ are the same.
\end{result}

The next result is a special case of a general result by Thai--Thu. We just state what will be
needed in this paper. In fact, in proving Theorem~\ref{T:conv_to_1},
\cite[Proposition~2.1]{berteloot:cmC2tag94} by Berteloot suffices. However, in 
Section~\ref{S:ques}, we shall discuss certain generalisations of Theorem~\ref{T:conv_to_1}
to $\C^n$ for all $n\geq 2$ (see Question~\ref{Q:corank_1} and
Question~\ref{Q:convex}). It is in this context that we wish to state a result that substitutes
an analytical condition in \cite[Proposition~2.1]{berteloot:cmC2tag94} by a geometric condition.
This result is as follows:  

\begin{result}[a paraphrasing of {\cite[Proposition~2.2]{thaiThu:cdCnnag09}}]\label{R:normal_conv}
Let $\ome$ be a domain in $\C^n$, $n\geq 2$, and let $\xi\in \bdy\ome$. Suppose $\bdy{\ome}$ is
$\smoo^\infty$-smooth and
Levi-pseudoconvex near $\xi$, and $\xi$ is of finite type. Let $\OM$ be a domain in $\C^n$.
A sequence $\{\varPsi_{\nu}\}\subset {\rm Hol}(\OM, \ome)$ converges uniformly on compact
subsets to $\xi$ if and only if $\lim_{\nu\to \infty}\varPsi_{\nu}(a) = \xi$ for some $a\in \OM$.
\end{result}

We conclude this section with a basic lemma. To state it, we introduce notation that will be convenient 
for presenting some of the proofs below. If $\OM\varsubsetneq \C^n$ is a domain, $\xi\in \bdy\OM$,
and a $\bdy{\OM}$-open neighbourhood of $\xi$ is $\smoo^1$, then $\tgt_{\xi}(\bdy\OM)$ will denote the 
real tangent hyperplane to $\bdy\OM$ at $\xi$ viewed \textbf{extrinsically} in $\C^n$. Specifically,
$\tgt_{\xi}(\bdy{\OM})$ is the $\R$-affine hyperplane given by
\[
  \tgt_{\xi}(\bdy\OM)\,:=\,\xi+\{x\in \R^{2n} : \big( x\mid \nabla\rho(\xi)\big) = 0\},
\]
where we identify $\C^n$ with $\R^{2n}$, $( \bcdot\,\mid \bcdot)$ is the standard inner product
on $\R^{2n}$, and $\rho$ is any ($\smoo^1$-smooth) defining function of the part of $\bdy\OM$ around $\xi$
that is $\smoo^1$-smooth.

\begin{lemma}\label{L:tangent}
Let $\OM$ be a bounded domain
and let $\xi\in \bdy{\OM}$. Let $B_{\xi}$ be an open ball with centre $\xi$ such that $\bdy{\OM}\cap B_{\xi}$
is a $\smoo^1$-smooth strictly convex
hypersurface. Suppose $\tgt_{\xi}(\bdy{\OM})\cap \ombar = \{\xi\}$. Then, there exists an
open ball $B^{\prime}_{\xi}\Subset B_{\xi}$ with centre $\xi$ such that
$\tgt_{q}(\bdy{\OM})\cap \ombar = \{q\}$ for every $q\in \bdy{\OM}\cap B^{\prime}_{\xi}$.
\end{lemma}

\begin{remark}
Refer to Definition~\ref{D:well_cnvx} for the meaning of the condition ``$\bdy{\OM}\cap B_{\xi}$ is a 
$\smoo^1$-smooth strictly convex hypersurface''.
\end{remark}

\begin{proof}
Since $\bdy{\OM}\cap B_{\xi}$ is a strictly convex hypersurface, there exists
a connected $\bdy{\OM}$-open set $\omega\Subset \bdy{\OM}\cap B_{\xi}$ containing $\xi$ and
a constant $r>0$ such that 
\begin{equation}\label{E:local_miss}
  \tgt_{q}(\bdy{\OM})\cap B^n(q, r)\cap \ombar\,=\,\{q\} \; \; \forall q\in \overline{\omega}.
\end{equation}
Also, as $\OM$ is bounded, $\exists R>0$ such that
\begin{equation}\label{E:far_miss}
  \big(\tgt_{q}(\bdy{\OM})\setminus \overline{B^n(q, R)}\,\big)\cap \ombar\,=\,\varnothing \; \;
  \forall q\in \bdy{\OM}\cap B_{\xi}.
\end{equation}
Let $(X_1,\dots, X_{2n-1})$ denote a local orthonormal frame for $\left. T(\bdy{\OM})\right|_{\bdy{\OM}\cap B_{\xi}}$.
Consider the function $\Phi: \overline{\omega}\to [0, +\infty)$ defined by
\[
  \Phi(q)\,:=\,\inf_{s\in \ombar}\,\inf_{x\in \R^{2n-1}:\,r\leq \|x\|\leq R}
  			\left\|\Big(q+\sum\nolimits_{k=1}^{2n-1}x_kX_k(q)\Big) - s\right\| \; \; \forall q\in \overline{\omega}.
\]
As $\bdy\OM\cap B_{\xi}$ is a $\smoo^1$-smooth hypersurface, the function
\[
  \overline{\omega}\times\ombar\times\{x\in \R^{2n-1}: r\leq \|x\|\leq R\}\ni
  (q, s, x)\longmapsto \left\|\Big(q+\sum\nolimits_{k=1}^{2n-1}x_kX_k(q)\Big) - s\right\|
\]
is clearly continuous. As the factors constituting the domain of the latter function are compact, it
is elementary that $\Phi$ is continuous. As $\tgt_{\xi}(\bdy{\OM})\cap \ombar = \{\xi\}$,
$\Phi(\xi)>0$. Thus, there exists an
open ball $B^{\prime}_{\xi}\Subset B_{\xi}$ with centre $\xi$ such that
$B^{\prime}_{\xi}\cap \bdy{\OM}\subseteq \omega$ and such that $\Phi(q) > 0$ for
every $q\in \bdy{\OM}\cap B^{\prime}_{\xi}$. Equivalently:
\[
  \tgt_{q}(\bdy{\OM})\cap \big(\overline{B^n(q, R)}\setminus B^n(q, r)\big)\cap \ombar\,=\,\varnothing
  \; \; \forall q\in \bdy{\OM}\cap B^{\prime}_{\xi}.
\]
This, together with \eqref{E:local_miss} and \eqref{E:far_miss} establishes the result.
\end{proof}

\section{A summary of scaling results of Bedford--Pinchuk}\label{S:scale}
As one would infer from the discussion in Subsection~\ref{SS:boubdary}, the proof of
Theorem~\ref{T:conv_to_1} relies on the scaling method. We present a short survey of the
method pioneered by Pinchuk (in the setting of results jointly with Bedford). To begin with, however,
we need a few preparatory statements.  
 
\begin{definition}\label{D:conv_Hausdorff}
Let $\{D_{\nu}\}$ be a sequence of open subsets of $\C^n$, and let $D\subset \C^n$ be open. We say that
\emph{$\{D_{\nu}\}$ converges to $D$ in the Hausdorff sense} (denoted by $D_{\nu}\to D$) if
the following conditions hold:
\begin{itemize}[leftmargin=24pt]
  \item For each compact set $K\subset D$, there exists a number $N\equiv N(K)$ such that
  $K\subset D_{\nu}$ for every $\nu\geq N$.  
  \item If a compact set $K$ is contained in $D_{\nu}$ for every sufficiently large $\nu$, then $K\subset D$.
\end{itemize}
\end{definition}

For a pair $(D,p)$ with the properties mentioned in Theorem~\ref{T:conv_to_1}, the following
result provides global holomorphic coordinates $(z,w)$ relative to which $\bdy{D}$ (locally) around $p$
can be presented as the graph of a function of a specific form that is useful to work with.
The phrase ``canonical coordinates'' in Section~\ref{S:conv_to_1} refers to these coordinates.

\begin{result}[Bedford--Pinchuk, {\cite[Lemma~2.1]{bedfordPinchuk:dC2ngha89}}]\label{R:canon_form}
Let $D$ be a bounded domain in $\C^2$ and let $p\in D$. Suppose $\bdy{D}$ is $\smoo^\infty$-smooth and
Levi-pseudoconvex near $p$, and $p$ is of finite type. Then, there exist a number $m\in \Z_+$, a biholomorphic
map $\Psi : \C^2\to \C^2$ with $\Psi(p)=0$\,---\,where $\Psi$ is a composition of a $\C$-affine map with a polynomial
map\,---\,and a neighbourhood $U$ of $0$ such that each
$(z,w)\in \Psi(\bdy{D})\cap U$ satisfies the equation
\begin{equation}\label{E:canon_defn}
  \re{w} + \Big(\psi(z, \zbar) + \Rem_1(z) + (\im{w})\Rem_2(z) + \Rem_3(z, \im{w})\Big)\,=\,0,
\end{equation}
where the functions appearing on the left-hand side of this equation are as follows:
\begin{itemize}[leftmargin=24pt]
  \item $\psi$ is a subharmonic polynomial in $z$ and $\zbar$ that is homogeneous of degree $2m$ and has no
  harmonic terms.
  \item $\Rem_1(z) = O(|z|^{2m+1})$ and  $\Rem_2(z) = O(|z|)$ as $z\to 0$.
  \item $\Rem_3(z, \im{w}) = O(|z|\,|\im{w}|^2)$ as $(z, \im{w})\to 0$.
\end{itemize}
Furthermore, if $m\geq 2$, then $\Rem_2(z) = O(|z|^{m+1})$ as $z\to 0$.
\end{result}

The last sentence of the above result is stated somewhat more conservatively in
\cite[Lemma~2.1]{bedfordPinchuk:dC2ngha89}. However, after expressing the function within large
brackets in \eqref{E:canon_defn} as a Taylor polynomial in $z$, $\zbar$ and $\im{w}$ of degree $2m+1$
plus the associated $o(\|(z, \im{w})\|^{2m+1})$ remainder, precisely the argument for proving
\cite[Lemma~2.1]{bedfordPinchuk:dC2ngha89} gives the conclusion above.
\smallskip

Now, given a pair $(D,p)$ as in Result~\ref{R:canon_form} and with $\Psi$ and $U$ as given by this result,
it is a standard argument that there exists a neighbourhood $U^{\prime}\Subset U$ of $0$ and
a smoothly bounded pseudoconvex domain $\U_p\subset \Psi(D)$ such that
%\begin{itemize}[leftmargin=27pt]
\begin{enumerate}[leftmargin=27pt, label=(\Roman*)]
  \item\label{I:U_p_small} $U^{\prime}\cap \Psi(D)\varsubsetneq \U_p\varsubsetneq U\cap \Psi(D)$; and
  \item\label{I:U_p_bdy} $U^{\prime}\cap \bdy\Psi(D) = U^{\prime}\cap \bdy{\U_p}$.
\end{enumerate}
\textbf{Fix} a defining function $\varrho$ of $\bdy{\U_p}$ such that $\varrho$ in some open 
neighbourhood of $U^{\prime}\cap \bdy{\U_p}$ is given by the left-hand side of \eqref{E:canon_defn}.
\smallskip

The following result summarises a key construction by Bedford--Pinchuk in \cite{bedfordPinchuk:dC2ngha89}.
We clarify some notation appearing therein: \emph{$(Z,W)$ will denote the standard global holomorphic
coordinates on $\C^2$ given by its product structure.} 

\begin{result}[Bedford--Pinchuk, {\cite[\S2]{bedfordPinchuk:dC2ngha89}}]\label{R:scaling_core}
Let the pair $(D, p)$ be as in Result~\ref{R:canon_form}, let $\Psi$
and $m$ be as given by that result, and let $U^{\prime}$, $\U_p$ and $\varrho$ be as introduced above
Let $\{z_{\nu}\}\subset D$ be a sequence with $z_{\nu}\to p$. Write:
\begin{itemize}[leftmargin=24pt]
  \item $(a_{\nu,\,1}, a_{\nu,\,2}) := \Psi(z_{\nu})$, and
  \item let $\ahat_{\nu,\,2}$ be the unique complex number
  such that $(a_{\nu,\,1}, \ahat_{\nu,\,2})\in \Psi(\bdy{D})\cap U^{\prime}$ (for $\nu$ sufficiently
  large) and $\im{a_{\nu,\,2}} = \im{\ahat_{\nu,\,2}}$. 
\end{itemize} 
Define
$A_{\nu}(z, w)\,:=\,\big(z-a_{\nu,\,1}, e^{-i\theta_{\nu}}(w-\ahat_{\nu,\,2})-b_{\nu}(z-a_{\nu,\,1})\big)$,
where $b_{\nu}\in \C$ and $\theta_{\nu}\in [-\pi, \pi)$ are such that 
\[
  T_0\big(A_{\nu}\circ\Psi(\bdy{D})\big) = \{(Z, W)\in \C^2 : \re{W}=0\}.
\]
Let $B_{\nu}$
be that polynomial shear such that if
$\varrho_{\nu} := \big((B_{\nu}\circ A_{\nu})^{-1}\big)^*\varrho$ and if we write
\[
  \left.
  \varrho_{\nu}\right|_{B_{\nu}\circ A_{\nu}({\sf dom}(\varrho)\cap U^{\prime})}(Z,W)\,=\,\re{W}
  + \sum_{k=2}^{2m}\sigma_{\nu,\,k}(Z,\overline{Z}) + O(|Z|^{2m+1}) + O(|Z|\,|\im{W}|),
\]
where $\sigma_{\nu,\,k}(Z,\overline{Z})$ are $\R$-valued polynomials that are
homogeneous of degree $k$, $k=2,\dots, 2m$, then $\sigma_{\nu,\,2},\dots, \sigma_{\nu,\,2m}$
have no harmonic terms. Fix a sequence $\{\delta_{\nu}\}$ such that
\begin{itemize}[leftmargin=30pt]
  \item[$(\bullet)$] the coeffiicient of the greatest magnitude in the polynomial
  $\sum_{k=2}^{2m}\delta_{\nu}^k\sigma_{\nu,\,k}(Z, \overline{Z})$ is\,$\thickapprox 1$.
\end{itemize}
Write $\eps_{\nu}:= |a_{\nu,\,2} - \ahat_{\nu,\,2}|$, and define the scalings
\[
  \Delta(Z,W)\,:=\,(Z/\delta_{\nu}, W/\eps_{\nu}) \; \; \forall (Z,W)\in \C^2.
\]
Then, we conclude that:
%\begin{itemize}[leftmargin=27pt]
\begin{enumerate}[leftmargin=27pt, label=$(\alph*)$]
  \item\label{I:not_blow} $\sup_{\nu\in \Z_+}\eps_{\nu}^{-1}\delta_{\nu}^{2m} < \infty$.
  \vspace{1mm}
  
  \item\label{I:subseq_cvg} There exists a sequence $\{\nu_j\}\subset \Z_+$ such that, on each closed ball
  with centre $0\in \C^2$, the (tail of the) sequence $\big\{(\Delta_{\nu_j}^{-1})^*\varrho_{\nu_j}\big\}$
  converges in $\smoo^\infty$ to $\re{W} + P(Z, \overline{Z})$, where $P$ is a subharmonic polynomial in
  $Z$ and $\overline{Z}$ having no harmonic terms and satisfying $P(0, 0) = 0$.%
  \vspace{1mm}
  
  \item\label{I:Hauss} Writing $\Phi_{\nu} := \Delta_{\nu}\circ B_{\nu}\circ A_{\nu}$, 
  the sequence of
  domains $\big\{\Phi_{\nu_j}(\Psi(D)\cap U^{\prime})\big\}$ converges in the Hausdorff sense
  to the unbounded domain $\OM$, where
  \[
    \OM\,=\,\{(Z,W)\in \C^2 : \re{W} + P(Z, \overline{Z}) < 0\}.
  \]
\end{enumerate}
\end{result} 
 
Result~\ref{R:scaling_core} summarises an argument forming a part of \S2 and the constructions that constitute
a part of the proof of Lemma~2.2 in \cite{bedfordPinchuk:dC2ngha89}. However, since the domain $\U_p$
does not play a part in the latter arguments, two clarifications are in order:
\begin{itemize}[leftmargin=24pt]
  \item The maps $B_{\nu}$ that are required in the description of $\Phi_{\nu}$
  in part~\ref{I:Hauss} of the above result do not appear explicitly in
  \cite{bedfordPinchuk:dC2ngha89}. However, their need is mentioned in the
  footnote~(4) added to the proof of \cite[Lemma~2.2]{bedfordPinchuk:dC2ngha89}.
  
  \item The conclusions of Result~\ref{R:scaling_core} hold for \textbf{any} neighbourhood
  $0\in \C^2$ having the properties stated in the discussion leading up to \eqref{E:canon_defn}.
  The property~\ref{I:U_p_bdy} above and the description of $\varrho$ that follows it establish
  $U^{\prime}$ to be just such a neighbourhood. We need the domain $\U_p$
  to be able to use a result by Diederich--Fornaess needed to establish
  Lemma~\ref{L:varrhos} (specifically, see
  Remark~\ref{R:varrhos} below).
\end{itemize}

\begin{remark}\label{R:sims}
At various points, our discussion will involve comparing quantities that depend on several parameters
(e.g., $(\bullet)$ in Result~\ref{R:scaling_core}). In such
arguments, we shall use the notation $X\lesssim Y$ to denote that there exists a constant $C>0$ that is 
\textbf{independent} of all parameters occurring in $X$ or $Y$ such that $X\leq CY$. The notation
$X\thickapprox Y$ will indicate that $X\lesssim Y$ and $X\gtrsim Y$.
\end{remark}

Before stating the next result, we should clarify that if $S\varsubsetneq \C^n$ and $z\notin S$,
then ${\rm dist}(z, S)$ has the following meaning:
\[
  {\rm dist}(z, S)\,:=\inf\{\|z-x\| : x\in S\}.
\]  

\begin{lemma}\label{L:n_nu}
Let the pair $(D, p)$ be as in Result~\ref{R:canon_form}, let $\Psi$ and $m$ be as given by that result, and
let $\{z_\nu\}\subset D$ be a sequence with
paraboloidal approach to $p$. Let $m\geq 2$. For each $\nu$, let $b_{\nu}$ be the 
coefficient introduced in Result~\ref{R:scaling_core}. Then,
\[
  |b_{\nu}|\,\lesssim\,{\rm dist}(z_{\nu}, \bdy{D}) \; \; \forall \nu\in \Z_+.
\]
\end{lemma}
\begin{proof}
Since $\Psi$ is a biholomorphism of $\C^2$ onto itself with the properties given by Result~\ref{R:canon_form},
\begin{equation}\label{E:model_equiv}
  {\rm dist}(z_{\nu}, \bdy{D})\,\thickapprox\,{\rm dist}\big(a_{\nu}, \bdy\Psi(D)\cap U^{\prime}\big)
  \; \; \forall \nu \text{ sufficiently large,}
\end{equation}
where $a_{\nu}$ and $U^{\prime}$ are as introduced above. For the same reason, and by the fact that
$\im{\ahat_{\nu,\,2}} = \im{a_{\nu,\,2}}$, the paraboloidal approach of $\{z_{\nu}\}$ translates to
\begin{multline}\label{E:order_dist}
  |a_{\nu,\,1}|\,\lesssim\,{\rm dist}\big(a_{\nu}, \bdy\Psi(D)\cap U^{\prime}\big)^{1/2} \quad \text{and} \\
  |\im{\ahat_{\nu,\,2}}|\,\lesssim\,{\rm dist}\big(a_{\nu}, \bdy\Psi(D)\cap U^{\prime}\big)^{1/2}
  \quad \forall \nu \text{ sufficiently large.}
\end{multline}
We now compute to find that the maximal
complex subspace of $T_{(a_{\nu, 1}, \ahat_{\nu, 2})}\big(\bdy\Psi(D)\cap U^{\prime}\big)$ (which is a
complex line) is spanned by the vector
\[
  \left( 1,
  \left. -\frac{\partial{\varrho}/\partial{z}}{\partial{\varrho}/\partial{w}}\right|_{z=a_{\nu, 1}, \ w= \ahat_{\nu, 2}}\right),
\]
where we recall that, in some fixed open 
neighbourhood of $\bdy\Psi(D)\cap U^{\prime}$, $\varrho$ is (owing to the property~\ref{I:U_p_bdy} above) given
by the left-hand side of \eqref{E:canon_defn}. Observe that, by the properties of $\Rem_3$ given by
Result~\ref{R:canon_form}, we may assume, by shrinking $U^{\prime}$ if necessary, that
\begin{itemize}[leftmargin=24pt]
  \item $\partial_{w}\varrho(a_{\nu,\,1}, \ahat_{\nu,\,2})\neq 0$ for every $a_{\nu}\in U^{\prime}$; and
  \item $|\partial_{w}\varrho(a_{\nu,\,1}, \ahat_{\nu,\,2})|\thickapprox 1$ corresponding to the above $a_{\nu}$.
\end{itemize}
Clearly, combining the definition of $b_{\nu}$ with the last two bullet-points, we have
\[
  |b_{\nu}|\,\thickapprox\,|\partial_{z}\varrho(a_{\nu,\,1}, \ahat_{\nu,\,2})|
  \; \; \forall \nu \text{ sufficiently large.}
\]
We now use the condition that $m\geq 2$. The latter estimate, in view of Result~\ref{R:canon_form}, 
\eqref{E:order_dist} and \eqref{E:model_equiv}, yields the result.
\end{proof}

We now present the lemma alluded to soon after Result~\ref{R:canon_form}.

\begin{lemma}\label{L:varrhos}
In the notation introduced in Result~\ref{R:canon_form}, and referring to the sequence $\{\nu_j\}$ and the
polynomial $P$ given by the latter result, let
us write
\begin{align*}
  \varrho^{(j)}\,&:=\,(\Delta_{\nu_j}^{-1})^*\varrho_{\nu_j}, \\
  h_j\,&:=\,(\Delta_{\nu_j}\circ B_{\nu_j}\circ A_{\nu_j})^{-1}, \\
  \rho(Z,W)\,&:=\,\re{W} + P(Z, \overline{Z}).
\end{align*}
Then, there exists a constant $\delta>0$ and a smooth function $\alpha$ defined on 
a neighbourhood of $\overline{\U}_p$, satisfying $\alpha(0)=0$, such that:
\begin{enumerate}[leftmargin=27pt, label=$(\alph*)$]
  \item\label{I:tildaed} The functions $\widetilde{\varrho}^{(j)}:=-e^{\alpha\circ h_j}\bcdot(-\varrho^{(j)})^\delta$
  are plurisubharmonic on $\Delta_{\nu_j}\circ B_{\nu_j}\circ A_{\nu_j}(\U_p)$.
  \item\label{I:tail_cvgs} On each closed ball with centre $0\in \C^2$, the (tail of the) sequence $\{\widetilde{\varrho}^{(j)}\}$
  converges in $\smoo^2$ to $-(-\rho)^{\delta}$. 
\end{enumerate}
\end{lemma}
    
\begin{remark}\label{R:varrhos}
The proof of Lemma~\ref{L:varrhos} is the argument on p.~\!172
of \cite{bedfordPinchuk:dCn+1nag91} by Bedford--Pinchuk\,---\,which concerns a
class of domains in $\C^{n+1}$, $n\geq 1$\,---\,with the parameter $n=1$. Thus,
we shall not repeat this argument. We do, however, indicate one difference.
The domain to which the argument in \cite{bedfordPinchuk:dCn+1nag91} applies is
pseudoconvex, which isn't necessarily so with $D$ in Theorem~\ref{T:conv_to_1}.
We instead have an auxiliary pseudoconvex domain $\U_p$ for which we fix the defining
function $\varrho$ with the properties stated right after \ref{I:U_p_small} and \ref{I:U_p_bdy} above. We
apply the main theorem of Diederich--Fornaess in \cite{diederichFornaess:pdbspef77}
to $\U_p$, after which the steps in the proof
are exactly as in \cite{bedfordPinchuk:dCn+1nag91}.
\end{remark}

\section{The proof of Theorem~\ref{T:conv_to_1}}\label{S:conv_to_1}

The proof of Theorem~\ref{T:conv_to_1} depends on a key lemma with which this section begins.
Before stating and proving this lemma it would be useful to introduce one piece of notation.
Let $\OM\varsubsetneq \C^n$ be a domain and $z\in \OM$. We define
\[
  d_{\OM}(z)\,:=\,{\rm dist}(z, \bdy{\OM}).
\]
We shall liberally use the notations explained in Remark~\ref{R:sims}.

\begin{lemma}\label{L:conv_to_1_key}
Let the pair $(D,p)$ be as in Theorem~\ref{T:conv_to_1}, and let $\{z_\nu\}\subset D$ be a
sequence with paraboloidal approach to $p$. Let $\OM$ be any limit domain (in the Hausdorff
sense) obtained by applying Result~\ref{R:scaling_core} to $\{z_{\nu}\}$ and let $\rho$ denote
the defining function for it described by Result~\ref{R:scaling_core}-\ref{I:subseq_cvg}. Let us
write
\[
  P(Z, \ov{Z})\,=\,\sum_{k=2}^{2m}\,\sum_{l=1}^{k-1}C_{k,\,l}Z^l\ov{Z}^{k-l}.
\]
Assume $m\geq 2$. Then $C_{2,\,1}=0$.
\end{lemma}
\begin{proof}
We adopt the notation introduced in the statement of Result~\ref{R:scaling_core} in its entirety.
If $(z,w)$ denotes the canonical coordinates centered at $p$ given by Result~\ref{R:canon_form}, then
the function $\psi$ given by the latter result has the form
\begin{equation}\label{E:sh_canon}
  \psi(z, \zbar)\,=\,\sum_{l=1}^{2m-1}c_{l}z^l\zbar^{2m-l}.
\end{equation}
We now analyse the canonical defining function given by \eqref{E:canon_defn} to conclude
that $\varrho_{\nu}(Z,W)$\,---\,as introduced by Result~\ref{R:scaling_core}\,---\,has the form
\begin{align}
  \varrho_{\nu}(Z,W)\,=&\,\re{W} + Q_{\nu}(Z) +
  \left[\sum_{l=1}^{2m-1}c_l(Z+a_{\nu,\,1})^l(\ov{Z}+\ov{a}_{\nu,\,1})^{2m-l}\right]
  + \Rem_1(Z+a_{\nu,\,1}) \notag \\
  +&\,\im(e^{i\theta_{\nu}}(W\!+\!b_{\nu}Z)+\ahat_{\nu,\,2})\Rem_2(Z+a_{\nu,\,1}) +
  \Rem_3\big(Z+a_{\nu,\,1}, \im(e^{i\theta_{\nu}}(W\!+\!b_{\nu}Z)+\ahat_{\nu,\,2})\big) \label{E:rho_nu} \\
  \equiv&\,\re{W} + \sum_{k=2}^{2m}\,\sum_{l=1}^{k-1}C^{(\nu)}_{k,\,l}Z^l\ov{Z}^{k-l}
  + O(|Z|^{2m+1}) + O(|Z|\,|\im{W}|) \notag \\
  & \hspace{0.5\textwidth} \forall (Z,W)\in B_{\nu}\circ A_{\nu}({\sf dom}(\varrho)\cap U^{\prime}), \notag
 \end{align}
where $(Z,W)$ denotes the standard global holomorphic coordinates on $\C^2$ given by its product structure,
$Q_{\nu}\in \C[Z]$ and whose role will be explained below,
and $\Rem_i$, $i=1, 2, 3$, are as described by Result~\ref{R:canon_form}. It might be helpful to
clarify/reiterate a couple of points.
\begin{itemize}[leftmargin=24pt]
  \item The coordinate system $(Z,W)$ is, in general, different from the canonical coordinates chosen
  to start off this proof.
  \item The polynomial $Q_{\nu}$ represents the difference between the function obtained by making
  the substitutions $z = Z+a_{\nu,\,1}$, 
  $w = e^{i\theta_{\nu}}(W\!+\!b_{\nu}Z)+\ahat_{\nu,\,2}$ in $\varrho$, and the
  function in the line following \eqref{E:rho_nu}; it therefore encodes the effect of the automorphism $B_{\nu}$. 
\end{itemize}
The proof of our lemma lies in estimating $|C^{(\nu)}_{2,1}|$. Different terms on the right-hand side of
\eqref{E:rho_nu} contribute to the coefficient $C^{(\nu)}_{2,1}$. We shall estimate these contributions
separately. Before deriving our key estimates, we note a pair of estimates.
Since, by hypothesis, $\{z_\nu\}\subset D$ is a sequence with paraboloidal approach to $p$, the
estimates \eqref{E:model_equiv} and \eqref{E:order_dist} are applicable and give the following:
\begin{equation}\label{E:a_bounds}
 |a_{\nu,\,1}|\,\lesssim\,d_D(z_{\nu})^{1/2} \quad \text{and}
 \quad |\im{\ahat_{\nu,\,2}}|\,\lesssim\,d_D(z_{\nu})^{1/2}
 \quad \forall \nu \text{ sufficiently large.} 
\end{equation}

It will be implicit that the estimates in this and the subsequent paragraphs involving the
quantity $d_D(z_{\nu})$ hold for
\textbf{all $\boldsymbol{\nu}$ sufficiently large.}
From the first inequality in \eqref{E:a_bounds} and by \eqref{E:rho_nu}, we have
\begin{equation}\label{E:first}
  \left. \begin{array}{r}
  		\text{the size of the coefficient of $Z\ov{Z}$ contributed} \\
		\text{by the sum in square brackets in \eqref{E:rho_nu}}
		\end{array}\right\}\,\lesssim\,d_D(z_{\nu})^{m-1}.
\end{equation}

To understand the contribution to the coefficient $C^{(\nu)}_{2,1}$ by the last three terms
in \eqref{E:rho_nu}, we must keep in mind that the expression of $\varrho$ comprises a Taylor polynomial,
where the degree of a non-zero monomial in $z$, $\zbar$ and $\im{w}$ is less than or equal to $2m$, plus
the associated remainder, which is $O(\|(z,\im{w})\|^{2m+1})$. With this in mind, the contribution to the
$Z\ov{Z}$-term (in the Taylor expansion) of $\varrho_{\nu}$ by the penultimate term
in \eqref{E:rho_nu} is a sum of monomials whose magnitudes are
\begin{align*}
  &\thickapprox\,|\im{\ahat_{\nu,\,2}}||a_{\nu,\,1}|^{\tau-2}|Z|^2, \; \; m+1\leq \tau\leq 2m, \text{ or } \\
  &\thickapprox\,|b_{\nu}||a_{\nu,\,1}|^{\tau-1}|Z|^2, \; \; m+1\leq \tau\leq 2m,
\end{align*}
where $\tau\geq m+1$ owing to our assumption that $m\geq 2$ and Result~\ref{R:canon_form}.
In view of Lemma~\ref{L:n_nu} and \eqref{E:a_bounds}, this implies:
\begin{equation}\label{E:third}
  \left. \begin{array}{r}
  		\text{the sizes of the coefficients of the above} \\
		\text{monomials contributing to the $Z\ov{Z}$-term of $\varrho_{\nu}$}
		\end{array}\right\}\,\lesssim\,d_D(z_{\nu})^{m/2}.
\end{equation} 
In view of the observation at the beginning of this paragraph (together with the order-of magnitude
estimate for $\Rem_3$ given by Result~\ref{R:canon_form}, which is why $s\geq 2$ below), the
contribution to the $Z\ov{Z}$-term
(in the Taylor expansion) of $\varrho_{\nu}$ by the last term in \eqref{E:rho_nu} is a
sum of monomials whose magnitudes are
\begin{align*}
  \left. \begin{array}{rl}
  		\thickapprox\,|b_{\nu}|^2|a_{\nu,\,1}|^{\tau}|\im{\ahat_{\nu,\,2}}|^{s-2}|Z|^2, 
  		& s\geq 2, \; \; \tau\geq 0, \text{ or } \\
  		\thickapprox\,|b_{\nu}||a_{\nu,\,1}|^{\tau-1}|\im{\ahat_{\nu,\,2}}|^{s-1}|Z|^2, 
  		& s\geq 2, \; \; \tau\geq 1, \text{ or } \\
  		\thickapprox\,|a_{\nu,\,1}|^{\tau-2}|\im{\ahat_{\nu,\,2}}|^{s}|Z|^2, 
  		& s\geq 2, \; \; \tau\geq 2,
  		\end{array}\right\} \; \text{where $s+\tau\leq 2m+1$}.
\end{align*}
In view of Lemma~\ref{L:n_nu} and \eqref{E:a_bounds} once again, this implies:
\begin{equation}\label{E:fourth}
  \left. \begin{array}{r}
  		\text{the sizes of the coefficients of the above} \\
		\text{monomials contributing to the $Z\ov{Z}$-term of $\varrho_{\nu}$}
		\end{array}\right\}\,\lesssim\,d_D(z_{\nu}).
\end{equation}  		
By an analogous (but much simpler) argument:
\begin{equation}\label{E:second}
  \left. \begin{array}{r}
  		\text{the size of the coefficient of the $Z\ov{Z}$-term of $\varrho_{\nu}$} \\
		\text{contributed by the third-from-last term in \eqref{E:rho_nu}}
		\end{array}\right\}\,\lesssim\,d_D(z_{\nu})^{m-\frac{1}{2}}.
\end{equation}

Recall that, by hypothesis, $m\geq 2$. Thus, from \eqref{E:first}--\eqref{E:second} and
recalling the meaning of 
$\delta_{\nu}$ and $\eps_{\nu}$ from Result~\ref{R:scaling_core}, we see that
$\eps_{\nu}\thickapprox d_D(z_{\nu})$ as $\nu\to \infty$, and
\begin{equation}\label{E:summary}
  \left. \begin{array}{r}
  		\text{the size of the coefficient of $Z\ov{Z}$ } \\
		\text{in (the Taylor expansion of) $\big(\Delta_{\nu}^{-1}\big)^*\varrho_{\nu}$}
		\end{array}\right\}\,\lesssim\,\frac{d_D(z_{\nu})\delta_{\nu}^2}{\eps_{\nu}}\,\thickapprox\,\delta_{\nu}^2.
\end{equation}
Owing to Result~\ref{R:scaling_core}-\ref{I:not_blow}, $\delta_{\nu}\to 0$ as $\nu\to \infty$. 
Owing to Result~\ref{R:scaling_core}-\ref{I:subseq_cvg},  the coefficients of $Z\ov{Z}$
in (the Taylor expansion of) $\big(\Delta_{\nu_j}^{-1}\big)^*\varrho_{\nu_j}$ converge
to $C_{2,\,1}$ as $j\to \infty$. From the last two facts and \eqref{E:summary} it follows
that $C_{2,\,1}=0$.		
\end{proof}

The following result will play a supporting role in the proof of Theorem~\ref{T:conv_to_1}. In
this result, and further in this section, $\kob_{\ome}$ will denote the Kobayashi pseudometric on a 
domain $\ome\subseteq \C^n$.

\begin{result}[paraphrasing {\cite[Proposition~6]{sibony:chm81}}]\label{R:koba_metric_lower}\label{R:Sibony}
Let $\ome$ be a domain in $\C^n$ and let $z \in \ome$. There exists a universal constant $\alpha > 0$ such
that if $u$ is a negative plurisubharmonic
function that is of class $\smoo^2$ in a neighbourhood of $z$ and satisfies
\[
  \big\langle v, (\cHess u)(z)v\big\rangle\,\geq\,c\|v\|^2 \quad \forall v\in \C^n,
\]
where $c$ is some positive constant, then 
\[
  \kob_{\ome}(z; v)\,\geq\,\left(\frac{c}{\alpha}\right)^{1/2}\!\frac{\|v\|}{|u(z)|^{1/2}}.
  \]
\end{result}

\noindent{Here $\langle\bcdot\,,\bcdot\rangle$ denotes the standard Hermitian inner product and $\cHess$ denotes the
complex Hessian. The objective of \cite{sibony:chm81} is to construct a pseudometric on $T^{1,0}\ome$\,---\,which
is known nowadays as the Sibony pseudometric\,---\,that is dominated by the Kobayashi pseudometric. The
lower bound in Result~\ref{R:Sibony} is actually a lower bound for the Sibony pseudometric, which results
in the the lower bound for $\kob_{\ome}(z; \bcdot)$.}
\smallskip

We are now in a position to provide:

\begin{proof}[The proof of Theorem~\ref{T:conv_to_1}]
Let $\{\nu_j\}$ be the sequence introduced by Result~\ref{R:scaling_core}-\ref{I:subseq_cvg}. We can find a subsequence,
which, without loss of generality, we can relabel as $\{\nu_j\}$ such that
\begin{equation}\label{E:rise}
  s_D(z_{\nu_1})\,<\,s_D(z_{\nu_2})\,<\,s_D(z_{\nu_3})<\dots<\,1.
\end{equation}
Let us write $\ome:=\Psi(D)$, where $\Psi$ is the
biholomorphic map, $\Psi: \C^2\to \C^2$, introduced by Result~\ref{R:canon_form}.
Then (writing $a_{\nu}:=\Psi(z_{\nu})$, $\nu\in \Z_+$), by definition
\begin{equation}\label{E:same}
  s_D(z_{\nu})\,=\,s_{\ome}(a_{\nu}) \; \; \text{for $\nu=1, 2, 3,\dots$}\,.
\end{equation}
Let $F_j: \ome\to B^2(0,1)$ be a holomorphic embedding such that $F_j(a_{\nu_j}) = 0$ and
$B^2\big(0, s_D(z_{\nu_j})\big)\subset F_j(\ome)$. That such an embedding exists
is the content of
\cite[Theorem~2.1]{dengGuanZhang:spsfbd12}\,---\,keeping \eqref{E:same} in mind.
As in the previous proof, we adopt the notation introduced in stating Result~\ref{R:scaling_core}. 
\smallskip

The core of our proof comprises the three claims below. For the theorems cited in Section~\ref{S:intro}, each
proof that relies on the scaling method features a version of either Claim~1 or~2 below.
In proofs featuring a version of Claim~1, its statement is closely tied to the the sequence $\{z_{\nu}\}$
being taken to lie in the inward normal to $\bdy{D}$ at $p$. For this reason, and as the argument for our
Claim~1 is short, we give a complete proof. Claim~2 merits a careful argument since earlier
arguments seem to rely on a stability theorem for the Sibony metric, which isn't known. In its place, we have
a simpler argument that just uses Result~\ref{R:Sibony}.
With these words, we state:
%\medskip
\pagebreak

\noindent{\textbf{Claim~1.} \emph{$F_j(\ome\cap U^\prime)\to B^2(0,1)$ in the Hausdorff sense.}}
\vspace{0.65mm}

\noindent{To see this, fix a $J\in \Z_+$. Then, by \eqref{E:rise},
\[
  B^2(0, s_J)\subset B^2(0, s_j)\subset F_j(\ome) \; \; \forall j\geq J,
\]
where we abbreviate $s_j:=s_D(z_{\nu_j})$.
Thus, we can define $\varPsi_{j,\,J}:= \left. F_j^{-1}\right|_{B^2(0,s_J)}$ for each $j\geq J$.
By construction, $\lim_{j\to \infty}\varPsi_{j,\,J}(0) = \lim_{j\to \infty}a_{\nu_j} = 0$. It
follows from Result~\ref{R:normal_conv} that
\begin{equation}\label{E:lim_const}
  \varPsi_{j,\,J}\lrarw 0 \; \text{uniformly on compact subsets of $B^2(0,s_J)$.}
\end{equation}
Now, fix a compact $K\subset B^2(0,1)$. Then, since $s_j\nearrow 1$,
there exists a $J(K)\in \Z_+$ such that
$K\subset B^2(0, s_{J(K)})$. By \eqref{E:lim_const}, there exists
a $\wt{J}(K)\geq J(K)$ such that 
\begin{align*}
  \varPsi_{j,\,J(K)}(K)&\subset \ome\cap U^\prime \\
  \Rightarrow\,K&\subset F_j(\ome\cap U^\prime) \; \; \forall j\geq \wt{J}(K).
\end{align*}
Since ${\sf range}(F_j)\subseteq B^2(0,1)$ for each $j\in \Z_+$, the above is enough to
conclude that $F_j(\ome\cap U^\prime)\to B^2(0,1)$ in the Hausdorff sense. Hence the claim.}
\medskip

As above, we write:
$h_j\,:=\,(\Delta_{\nu_j}\circ B_{\nu_j}\circ A_{\nu_j})^{-1} :
  \Phi_{\nu_j}(\ome\cap U^\prime)\to \ome\cap U^\prime$.
Next, consider the maps
\[
  G_j\,:=\,F_j\circ h_j: \Phi_{\nu_j}(\ome\cap U^\prime)\to B^{2}(0,1), \; \; j=1,2,3,\dots \,.
\]
Recall that, by Result~\ref{R:scaling_core}-\ref{I:Hauss}, $\Phi_{\nu_j}(\ome\cap U^\prime)\to \OM$ in the
Hausdorff sense. Thus, as $B^2(0,1)$ is bounded, it follows from Montel's theorem that we may
assume\,---\,by passing to a subsequence and relabelling if necessary\,---\,that the sequence $\{G_j\}$
converges uniformly on compact subsets to a holomorphic map $G_0: \OM\to \overline{B^2(0,1)}$.
By the plurisubharmonicity of $\|G_0\|^2$ and by the maximum principle it follows that if
$G_0(\OM)\cap \bdy{B^2(0,1)}\neq \varnothing$, then $G_0(\OM)\subseteq \bdy{B^{2}(0,1)}$.
Let ${\sf Arg}\big(\partial_{w}\varrho(a_{\nu,\,1}, \ahat_{\nu,\,2})\big)$ denote a number in
$(-\pi, \pi]$ such that
\[
  \frac{\partial\varrho}{\partial w}(a_{\nu,\,1}, \ahat_{\nu,\,2})\,=\,\exp
  \left\{i{\sf Arg}\big(\partial_{w}\varrho(a_{\nu,\,1}, \ahat_{\nu,\,2})\big)\right\}.
\]
Then, it follows from the definition of $A_{\nu}$ in Result~\ref{R:scaling_core} that
$\theta_{\nu} = -{\sf Arg}\big(\partial_{w}\varrho(a_{\nu,\,1}, \ahat_{\nu,\,2})\big)$. Write
\[
  \tau_j\,:=\,\left(0, 
  e^{-i\theta_{\nu_j}}\frac{a_{\nu_j,2}-\ahat_{\nu_j,2}}{\eps_{\nu_j}}\right)\,=\,(0, -e^{-i\theta_{\nu_j}}),
  \; \; j=1,2,3,\dots \,.
\]
Then, by construction, $G_j(\tau_j) = (0,0)$ for each $j\in \Z_+$, and $\tau_j\to (0, -1)$. Thus, as $\{G_j\}$
converges uniformly on compacts, $G_0(0, -1) = (0,0)$, and so 
$G_0(\OM)\subseteq B^2(0,1)$. We now prove the following
\medskip

\noindent{\textbf{Claim~2.} \emph{$G_0$ is a biholomorphic map.}}
\vspace{0.65mm}

\noindent{In view of a classical result of H.~\!Cartan, it suffices to find a point $z\in \OM$ such that
$G_0^\prime(z)$ is non-singular. To do so, we turn to Lemma~\ref{L:varrhos}. The functions appearing
below are exactly as given by Lemma~\ref{L:varrhos}. Let us fix a point $q_0\in \OM$ close to
$(-1,0)\in \OM$ at which the function
\[
  \wt{\varrho}^{(0)}\,:=\,-(-\rho)^{\delta}
\]
(here $\delta>0$ is as given by Lemma~\ref{L:varrhos}) is strictly plurisubharmonic.
By part~\ref{I:tail_cvgs} of Lemma~\ref{L:varrhos}, there exists a $J\in \Z_+$ and a constant $c>0$ such that
\begin{align}
  \big\langle v, (\cHess\wt{\varrho}^{(j)})(q_0)v\big\rangle\,&\geq\,c\|v\|^2, \; \; \text{and}
  \label{E:Levi_uni}\\
  \frac{\wt{\varrho}^{(0)}(q_0)}{2}\,\geq\,\wt{\varrho}^{(j)}(q_0)\,&\geq\,\frac{3\wt{\varrho}^{(0)}(q_0)}{2}
  \quad \forall j\geq J \text{ and $v\in \C^n$}
  \label{E:psh-value}
\end{align}
(%here $\langle\bcdot\,,\bcdot\rangle$ denotes the standard Hermitian inner product and
recall that $\cHess$ denotes the
complex Hessian). Let us write $\ome_j:=\Phi_{\nu_j}(\ome\cap U^\prime)$.
We have already seen that $G_0(\OM)\subset B^2(0,1)$. So, as $G_j\to G$ converges uniformly on compact
sets, there exists a number $R_1\in (0,1)$ such that $\{G_j(q_0): j\in \Z_+\}\subset B^2(0,R_1)$. Fix
$R_2\in (R_1, 1)$. Since, by Claim~1, $G_j(\ome_j)\to B^2(0,1)$ in the Hausdorff sense, we may
assume, by raising the value of the $J$ in \eqref{E:psh-value} if necessary, that
\begin{equation}\label{E:sub_dom}
  B^2(0, R_2)\subset G_j(\ome_j) \; \; \forall j\geq J.
\end{equation}}

Owing to the semicontinuity properties of the Kobayashi metric, the number
\[
  C\,:=\,\sup_{x\in \ov{B^2(0,R_1)},\,\|V\|=1}\kob_{B^2(0,R_2)}(x;V)
\]
is finite. By Result~\ref{R:Sibony}, \eqref{E:Levi_uni}, \eqref{E:psh-value}
and the fact that each $G_j$ is biholomorphic, we can find a constant $\wt{c}>0$ such that
\begin{align}
  \kob_{G_j(\ome_j)}\big(G_j(q_0); G_j^\prime(q_0)v\big)\,&=\,\kob_{\ome_j}
  (q_0; v) \notag \\
  &\geq\,\wt{c}\,\sqrt{\frac{2}{3}}\,\frac{\|v\|}{|\wt{\varrho}^{(0)}(q_0)|^{1/2}}\,\equiv\,\wt{C}\|v\|
  \quad \forall v\in \C^n \text{ and } \forall j\geq J. \label{E:interim}
\end{align}
The equality above follows from the fact that each $G_j$ is a biholomorphic map.
The inequality follows by taking $u$ to be $\wt{\varrho}^{(j)}$\,---\,which, by 
Lemma~\ref{L:varrhos}, is a negative plurisubharmonic function on $\ome_j$\,---\,in
Result~\ref{R:Sibony}. By \eqref{E:sub_dom} and the fact that the inclusion
$B^2(0,R_2)\hookrightarrow G_j(\ome_j)$, $j\geq J$, is holomorphic, thus metric-decreasing,
\[
  \kob_{B^2(0,R_2)}\big(G_j(q_0), G_j^\prime(q_0)v\big)\,\geq\,\kob_{G_j(\ome_j)}\big(G_j(q_0); G_j^\prime(q_0)\big)
  \; \; \forall v\in \C^n \text{ and } \forall j\geq J.
\]
Recall that $\{G_j(q_0): j\in \Z_+\}\subset B^2(0,R_1)$. Thus, given the definition of $C$ above, we conclude from
the last inequality and \eqref{E:interim} that
\[
  \|G_j^\prime(q_0)v\|\,\geq\,\frac{\wt{C}}{C}\,\|v\| \; \; \forall v\in \C^n \text{ and } \forall j\geq J.
\]
Hence, since $G_j\to G_0$ uniformly on compact subsets, we get
\[
  \|G_0^\prime(q_0)v\|\,\geq\,\frac{\wt{C}}{C}\,\|v\| \; \; \forall v\in \C^n.
\]
Therefore $G_0^\prime(q_0)$ is non-singular. By the observations at the beginning of the
previous paragraph, we conclude that $G_0$ is biholomorphic.
\medskip

We now need to establish the third claim. Its proof is, essentially, the proof given in the sub-section
entitled \emph{Surjectivity of $\widehat{\sigma}$} in the proof
of \cite[Proposition~3.3]{jooKim:bpwsft118}. Therefore, after some explanatory statements,
the reader will be referred to the latter for the proof of the following
\medskip

\noindent{\textbf{Claim~3.} \emph{$G_0(\OM)=B^2(0,1)$.}}
\vspace{0.65mm}

\noindent{Assume that $G_0(\OM)\varsubsetneq B^2(0,1)$. Then, there is a point $q\in \bdy{G_0(\OM)}$
belonging to $B^2(0,1)$. Pick $r\in (\|q\|, 1)$. Since, by Claim~1, $G_j(\ome_j)\to B^2(0,1)$ in the
Hausdorff sense, there exists a $J\in \Z_+$ such that
$B^2(0, r)\subset G_j(\ome_j)$ for every $j\geq J$.
Recall that, by definition, $F_j^{-1}(0) = a_{\nu_j}\to 0$ as $j\to \infty$. As $0\in \bdy(\ome\cap U^\prime)$,
applying Result~\ref{R:normal_conv} to
\begin{itemize}[leftmargin=24pt]
  \item the sequence $\big\{F_j^{-1}|_{B^2(0,r)}\big\}$, and
  \item the target-space $(\ome\cap U^\prime)$,
\end{itemize}
we conclude that $F_j^{-1}|_{B^2(0,r)}\to 0$ uniformly on compact subsets. In particular
$F_j^{-1}(q)\to 0$ as $j\to \infty$. Write $q_j:= F_j^{-1}(q)$. At this stage, the argument in
\cite[Proposition~3.3]{jooKim:bpwsft118} applies \emph{mutatis mutandis} (with some of the
required changes being:
\begin{align*}
  G_0 \; \text{replacing $\widehat{\sigma}$} \; \; \; &\text{and} \; \; \; F_j \; \text{replacing $f_j$}, \\
  \OM \; \text{replacing $\widehat{\OM}$} \; \; \; &\text{and} \; \; \; \ome\cap U^\prime \; \text{replacing $\OM$},
\end{align*}
etc.) This is because the argument referred to is insensitive to the difference between the domain
$\widehat{\OM}$ therein and the domain $\OM$ (as given by 
Result~\ref{R:scaling_core}) which is the replacement of the latter. Hence, the claim follows.}
\medskip

From Claim~2 and Claim~3 we conclude that $B^2(0,1)$ is biholomorphic to $\OM$. Now, since,
via an appropriate Cayley transformation, $B^2(0,1)$ is biholomorphic to the Siegel half-space
$\mathcal{H}\,:=\,\{(Z,W)\in \C^2 : \re{W}+\|Z\|^2<0\}$, we have
\begin{equation}\label{E:biholo}
  \mathcal{H} \; \text{is biholomorphic to} \; \OM.
\end{equation}

Let us now assume that $p$ is a \textbf{weakly} pseudoconvex point. Therefore, $m\geq 2$ in
Result~\ref{R:canon_form}. By Lemma~\ref{L:conv_to_1_key}, therefore, $C_{2,\,1}=0$.
This implies that the degree of $\nonh{P}$ is at least $4$. But this contradicts that fact that\,---\,in view of \eqref{E:biholo}
and Result~\ref{R:oelj}\,---\,the degree of $\nonh{P}$ is $2$. Hence, our assumption must be false. Therefore,
$p$ is a strongly pseudoconvex point.
\end{proof}

We conclude this section by elaborating on an observation made in Remark~\ref{R:BS}. To do so,
we rely on the notation in Sections~\ref{S:scale} and~\ref{S:conv_to_1}.

\begin{remark}\label{R:BS_discuss}
We must explain a gap in the proof of Theorem~1.1 by $\hkim$ in the paper
\emph{A note on the boundary behaviour of the squeezing function and Fridman invariant}
(Bull. Korean Math. Soc. {\bf 57} (2020), pp.~1241--1249), which we referred to in Remark~\ref{R:BS}.
The class of (pointed) domains considered in
their work is a subclass of the pointed domains $(D,p)$ described in Question~\ref{Q:corank_1}. For simplicity,
we focus on the set-up when $n=2$ and let $D\Subset \C^2$, where the pair $(D, p)$ is as in
Theorem~\ref{T:conv_to_1} with the further condition that $D$ is globally pseudoconvex.
$\hkim$ consider a sequence of points $\{z_{\nu}\}\subset D$ with \textbf{unrestricted} approach
to $p\in \bdy{D}$. Their proof, resulting in the same conclusion as of Theorem~\ref{T:conv_to_1}, relies
on a scaling construction equivalent to the one given by Result~\ref{R:scaling_core}. Let
$m$, $P$ and $\OM$ be as introduced in the latter result. The relationship between the notation in
Result~\ref{R:scaling_core} and that in H.~\!Kim~\emph{et al.} is as follows:
\[
  \OM\,=\,M_P, \quad\text{and}
  \quad \eta^\prime_{\nu}\,=\,(a_{\nu,\,1},\ahat_{\nu,\,2}),
  \; \ \delta_{\nu}\,\thickapprox\,\tau(\eta^\prime_{\nu}, \eps_{\nu}) \; \; \forall \nu\in \Z_+,
\]
where the expressions on the right are the notations in $\hkim$ 
It is well known that
\begin{equation}\label{E:variatn}
  d_{\ome}(a_{\nu})^{1/2}\lesssim \delta_{\nu}\lesssim d_{\ome}(a_{\nu})^{1/2m}, \quad\text{equivalently} \quad
  \eps_{\nu}^{1/2}\lesssim \tau(\eta^\prime_{\nu}, \eps_{\nu})\lesssim \eps_{\nu}^{1/2m}
\end{equation}
for each $\nu\in \Z_+$ and that, depending on the sequence $\{a_{\nu}\}$,
$\tau(\eta^\prime_{\nu}, \eps_{\nu})$ can vary through the entire range
indicated by \eqref{E:variatn}. This is the reason that the degree of $P$ is$\,\leq 2m$ and
\textbf{not necessarily} $2m$. An illustration of this is provided by a sequence
$\{z_{\nu}\}$ such that
\[
  \|\pi_p(z_{\nu}-p)\|\,\thickapprox\,d_{D}(z_{\nu})^{1/s} \; \; \forall\nu \ \text{sufficiently large}
\]
(in the notation of Definition~\ref{D:parabol}) and $s$ is a number satisfying $s\gg 2m$, in which
case the $P$ obtained by the algorithm
in Result~\ref{R:scaling_core} (equivalently, in Section~2 of the paper by $\hkim$) has degree
less than $2m$! Now, the closing arguments for the result by $\hkim$ are:
\begin{itemize}[leftmargin=24pt]
  \item using the additional fact that $s_{D}(z_{\nu})\to 1$ to infer that $M_P$ is biholomorphic to
  $B^2(0,1)$ (which is also our strategy above: note that $\OM=M_P$);
  
  \item concluding, using a result in \cite{berteloot:cmC2tag94} along the sames lines as
  Result~\ref{R:oelj}, that $2m=2$.
\end{itemize}
The second step in the above argument presupposes that the degree of $P$ \textbf{equals} $2m$.
There is a gap in explaining why this is so when $\{z_{\nu}\}$ has \textbf{unrestricted} approach.
(Perhaps, the authors have a different scaling algorithm in mind when the approach of
$\{z_{\nu}\}$ to $p$ is ``very tangential''.) Of course, the conclusion of $\hkim$ implies that the degree
of $P$ must equal $2m$, but assuming so at the juncture given by the second bullet-point above results
in a circular argument! As a contrast: in the present work, we circumvent this difficulty by assuming that $m>1$ and
using  Result~\ref{R:oelj} to arrive at a contradiction.
\end{remark}

\section{The proof of Theorem~\ref{T:HHR}}\label{S:HHR}

Before we can give a proof of Theorem~\ref{T:HHR}, we must fix some notation. First, recall that
given a domain $\OM\varsubsetneq \C^n$ and $\xi\in \bdy\OM$ such that $\bdy\OM$ is $\smoo^1$-smooth
around $\xi$, 
$\tgt_{\xi}(\bdy\OM)$ is as defined in Section~\ref{S:prelim}: i.e., the \textbf{affine} real hyperplane
in $\C^n$ containing $\xi$ that is tangent to $\bdy{\OM}$ at $\xi$. 
Next, if $V$ is a $\C$-linear subspace of $\C^n$,
$w\in \C^n$, and $r>0$, then we define
\[
  B_{w}(V;r)\,:=\,\{z\in \C^n: z-w\in V \ \text{and} \ \|z-w\|<r\}.
\]

We also need the following proposition. The result is an exercise in coordinate geometry
and is the outcome of Steps~3 and~4 of the proof of \cite[Theorem~1.1]{kimZhang:uspbccdCn16}
by Kim--Zhang. The inner product on $\C^n$ used in the statement below is the standard Hermitian
inner product on $\C^n$.

\begin{proposition}[following {\cite[Theorem~1.1]{kimZhang:uspbccdCn16}}]\label{P:HHR_key}
Let $\OM\varsubsetneq \C^n$, $n\geq 2$, be a bounded domain. Let $\xi\in \bdy{\OM}$ and suppose
a $\bdy\OM$-open neighbourhood of $\xi$ is $\smoo^1$-smooth.
Suppose there exists an open ball $B^{\prime}_{\xi}$ with centre $\xi$
such that $\OM\cap B^{\prime}_{\xi}$ is convex, and
\begin{equation}\label{E:conv-like}
  \tgt_{q}(\bdy{\OM})\cap \ombar = \{q\} \; \; \forall q\in \bdy{\OM}\cap B^{\prime}_{\xi}.
\end{equation}
Then, there exists a constant $\eps_0>0$ such that for each point $w\in \OM$ with
$\|w-\xi\|<\eps_0$, there exist points $s_1\equiv s_1(w),\dots, s_n\equiv 
s_n(w)\in \bdy{\OM}\cap B^{\prime}_{\xi}$ with the following properties:
\begin{enumerate}[leftmargin=27pt, label=$(\alph*)$]
  \item\label{I:s_1} $s_1$ satisfies
   $\|s_1-w\|\,=\,\sup\{r>0 : B^n(w,r)\subset \OM\}$.
  \item\label{I:other_s} For $j=2,\dots,n$, $s_j$ satisfies
  \[
    \|s_j-w\|\,=\,\sup\left\{r>0 : 
    B_{w}\big(\big({\rm span}_{\C}\{s_1-w,\dots, s_{j-1}-w\}\big)^{\perp}; r\big)\subset \OM\right\}.
  \]
  \item\label{I:j_discs} $B_{w}({\rm span}_{\C}\{s_j-w\}; \|s_j-w\|)\subset \OM\cap B^{\prime}_{\xi}$ for
  each $j=1,\dots, n$.
\end{enumerate}
The $n$-tuple of vectors $(e_1(w),\dots, e_n(w))$ defined by
 $e_j(w)\,:=\,(s_j-w)/\|s_j-w\|$, $j=1,\dots, n$,
is an orthonormal tuple of vectors. 
Furthermore, if we define the affine map $\Lambda_w: \OM\to \C^n$ by
\[
  \Lambda_{w}(z)\,:=\,\sum_{j=1}^n\frac{\langle z-w, e_j(w)\rangle}{\|s_j-w\|}\,\bas_j
\]
(where $\bas_1,\dots, \bas_n$ constitute the standard basis over $\C$ of $\C^n$), then there
exist constants $a^{j,\,k}\equiv a^{j,\,k}(w)$, $2\leq j\leq n$ and $1\leq k\leq j$, such that
\begin{align*}
  \tgt_{\!\Lambda_{w}(s_j)}\big(\Lambda_{w}(\bdy{\OM})\big)\,&=\,\left\{(Z_1,\dots,Z_n):
  \re\left(\sum\nolimits_{k=1}^{j-1}a^{j,\,k}Z_k + a^{j,\,j}(Z_j-1)\right)=0\right\},
  \quad j=2,\dots,n, \\
  \tgt_{\!\Lambda_{w}(s_1)}\big(\Lambda_{w}(\bdy{\OM})\big)\,&=\,\left\{(Z_1,\dots,Z_n):
  \re(Z_1-1)=0\right\},
\end{align*}
and such that $\sum_{k=1}^j|a^{j,\,k}|^2=1$ and $a^{j,\,j}>0$. Lastly,
$\tgt_{\!\Lambda_{w}(s_j)}\big(\Lambda_{w}(\bdy{\OM})\big)\cap \ov{\Lambda_{w}(\OM)} = \{\bas_j\}$
for $j=1,\dots,n$. 
\end{proposition}

We are now in a position to give:

\begin{proof}[The proof of Theorem~\ref{T:HHR}]
Let us assume that $D$ is not holomorphic homogeneous regular. Then, there exists a
sequence $\{z_{\nu}\}\subset D$ and a point $p\in \bdy{D}$ such that
\begin{equation}\label{E:0_ass_1}
  \lim_{\nu\to \infty}z_{\nu}\,=\,p \quad\text{and} \quad
  \lim_{\nu\to \infty}s_{D}(z_{\nu})\,=\,0.
\end{equation}
If $p$ is a point at which $\bdy{D}$ is strongly Levi pseudoconvex, then it is a \emph{spherically
extreme boundary point} in the sense of \cite[Section~3]{kimZhang:uspbccdCn16} and, thus,
$\lim_{\nu\to \infty}s_{D}(z_{\nu}) =1$. This is impossible in view of \eqref{E:0_ass_1}. Therefore,
$p$ is a point at which $\bdy{D}$ is \textbf{not} strongly Levi pseudoconvex. Therefore, by
hypothesis, there exists a biholomorphic map $\Psi: D\to \C^n$ having the properties stated in
Definition~\ref{D:well_cnvx}. Let us write
\[
  \OM\,:=\,\Psi(D), \; \; w_{\nu}\,:=\,\Psi(z_{\nu}), \; \; \text{and}
  \; \; \xi\,:=\,\Psi(p).
\]
Furthermore, as $\Psi\in \smoo(D\cup \{p\}; \C^n)$, we have (since $s_{D}$ is preserved by biholomorphic maps)
\begin{equation}\label{E:0_ass_2}
  \lim_{\nu\to \infty}w_{\nu}\,=\,\xi \quad\text{and} \quad
  \lim_{\nu\to \infty}s_{\OM}(w_{\nu})\,=\,0.
\end{equation}
Since, by Definition~\ref{D:well_cnvx}, there is an open ball $B_{\xi}$ with centre $\xi$ such that
\begin{itemize}[leftmargin=24pt]
  \item $\bdy\OM\cap B_{\xi}$ is $\smoo^1$-smooth and a strictly convex hypersurface, and
  \item $\tgt_{\xi}(\bdy\OM)\cap \ov{\OM} = \{\xi\}$,
\end{itemize}
it follows from Lemma~\ref{L:tangent} that there is an open ball $B^{\prime}_{\xi}\Subset B_{\xi}$ concentric
with $B_{\xi}$ such that
\[
  \tgt_{q}(\bdy{\OM})\cap \ombar = \{q\} \; \; \forall q\in \bdy{\OM}\cap B^{\prime}_{\xi}.
\]
This is precisely the condition \eqref{E:conv-like}. Of course, $\OM\cap B^{\prime}_{\xi}$ is convex.
We can thus apply Proposition~\ref{P:HHR_key}.
Without loss of generality, we may assume that $\|w_{\nu}-\xi\|<\eps_0$ for every $\nu$, where
$\eps_0$ is as given by Proposition~\ref{P:HHR_key}.  Let us abbreviate:
\begin{align*}
  \big(e_1(\nu),\dots, e_n(\nu)\big)\,&:=\,\big(e_1(w_{\nu}),\dots, e_1(w_{\nu})\big), \quad
  \Lambda_{\nu}\,:=\,\Lambda_{w_{\nu}}, \\
  s_j(\nu)\,&:=\,s_j(w_{\nu}), \; \; \; j=1,\dots,n, \\
  a^{j,\,k}(\nu)\,&:=\,a^{j,\,k}(w_{\nu}), \; \; \; k=1,\dots, j, \text{ and } j=2,\dots,n.
\end{align*}
where the objects on the right-hand side above are as given by Proposition~\ref{P:HHR_key} taking
$w=w_{\nu}$, $\nu=1,2,3,\dots$\,. Let $(Z_1,\dots,Z_n)$ be the global coordinates provided by the
product structure of $\C^n$. By the definition of these objects, for each $\nu$:
\begin{multline*}
  \Lambda_{\nu}\big(B_{w_{\nu}}({\rm span}_{\C}\{e_j(\nu)\}; \|s_j(\nu)-w_{\nu}\|)\big) \\
  =\,\{(Z_1,\dots, Z_n):|Z_j|<1 \text{ and } Z_k=0 \ \forall k\neq j\}, \; \; \; j=1,\dots, n.
\end{multline*}
From the above, and by the property~\ref{I:j_discs} in Proposition~\ref{P:HHR_key}, we have for each $\nu$:
\begin{equation}\label{E:discs}
  \{(Z_1,\dots, Z_n):|Z_j|<1 \text{ and } Z_k=0 \ \forall k\neq j\}\subset \Lambda_{\nu}(\OM\cap B^{\prime}_{\xi})
  \; \; \; \text{for $j=1,\dots,n$.}
\end{equation}
At this stage, we appeal to the fact that $\OM\cap B^{\prime}_{\xi}$ is convex. Hence, as
$\Lambda_{\nu}(\OM\cap B^{\prime}_{\xi})$ is convex, we deduce from \eqref{E:discs} that
\begin{equation}\label{E:acorn}
  \acorn\,:=\,\{(Z_1,\dots, Z_n): |Z_1|+\dots+|Z_n|<1\}\subset \Lambda_{\nu}(\OM\cap B^{\prime}_{\xi})
  \; \; \; \forall \nu.
\end{equation}

The argument in this paragraph has been given in \cite{kimZhang:uspbccdCn16} under the assumption
that $\OM$ is convex. Since the conclusion of this paragraph is a crucial part of this proof, but as $\OM$ is
\textbf{not necessarily} convex, we revisit that argument here. By \eqref{E:acorn},
$B^n(0, 1/\sqrt{n})\subset \acorn\subset \Lambda_{\nu}(\OM\cap B^{\prime}_{\xi})$
for each $\nu$. Furthermore, by the final conclusion of Proposition~\ref{P:HHR_key}, 
\begin{itemize}[leftmargin=30pt]
  \item[$(\bloz)$]
  $\Lambda_{\nu}(\OM)$ lies in one of the open half-spaces determined by
  $\tgt_{\!\Lambda_{\nu}(s_j(\nu))}\big(\Lambda_{\nu}(\bdy{\OM})\big)$ for $j=1,\dots,n$, for
  each $\nu$.
\end{itemize}
Thus, for each $\nu$:
\[
  {\rm dist}\big(0, \tgt_{\!\Lambda_{\nu}(s_j(\nu))}\big(\Lambda_{\nu}(\bdy{\OM})\big)\big)\,\geq\,1/\sqrt{n},
  \; \; \; j=1,\dots,n.
\] 
From this, the expressions for $\tgt_{\!\Lambda_{\nu}(s_j(\nu))}\big(\Lambda_{\nu}(\bdy{\OM})\big)$ given by
Proposition~\ref{P:HHR_key} and the fact that $\sum_{k=1}^j|a^{j,\,k}|^2=1$, it follows
that for each $\nu$:
\[
  \left|\left. \re\left(\sum\nolimits_{k=1}^{j-1}a^{j,\,k}(\nu)Z_k
  + a^{j,\,j}(\nu)(Z_j-1)\right)\right|_{Z=0}\right|\,\geq\,1/\sqrt{n},
  \; \; \; j=2,\dots,n,
\]
whence we have
\begin{equation}\label{E:key_coeffs}
  a^{j,\,j}(\nu)\geq 1/\sqrt{n} \; \; \; \text{for $j=2,\dots,n$, and for every $\nu$}.
\end{equation}
   
Now, let us define the invertible linear maps
\[
  L_{\nu} : (Z_1,\dots, Z_n)\longmapsto \big(Z_1, a^{2,\,1}(\nu)Z_1+a^{2,\,2}(\nu)Z_2,\dots,
  (a^{n,\,1}(\nu)Z_1+\dots +a^{n,\,n}(\nu)Z_n)\big).
\]
Owing to $(\bloz)$, we have for each $\nu$:
\begin{equation}\label{E:orthant}
  L_{\nu}\circ\Lambda_{\nu}(\OM)\subset \Big(\bigcap\nolimits_{2\leq j\leq n}\{(Z_1,\dots, Z_n): \re{Z_j}<a^{j,\,j}(\nu)\}
  \Big)\bigcap \{(Z_1,\dots, Z_n): \re{Z_1}<1\}.
\end{equation}
Next, for each $\nu$, the map
\[
  \Phi_{\nu}: (Z_1,\dots, Z_n)\longmapsto \left(\frac{Z_1}{2-Z_1},\,\frac{Z_2}{2a^{2,\,2}(\nu)-Z_2},
  \dots,\,\frac{Z_n}{2a^{n,\,n}(\nu)-Z_n}\right)
\]
is holomorphic on the corresponding region on the right-hand side of \eqref{E:orthant} and maps it 
biholomorphically onto $\D^n$. Now, by \eqref{E:acorn}, \eqref{E:key_coeffs} and the expressions
for the maps $L_{\nu}$, we can find a constant $\delta_0>0$ which depends only on $n$ such that
\begin{equation}\label{E:incl1}
  D(0,\delta_0)^n\subset L_{\nu}(\acorn)\subset L_{\nu}\circ\Lambda_{\nu}(\OM) \; \; \; \forall \nu.
\end{equation}
Since
$\Phi_{\nu}\circ L_{\nu}\circ \Lambda_{\nu}(\OM)\subset \D^n$ for each $\nu$,
this, together with \eqref{E:incl1}, gives us
\begin{equation}\label{E:incl2}
  D\big(0, \delta_0/(2+\delta_0)\big)^n\subset \Phi_{\nu}\circ L_{\nu}\circ \Lambda_{\nu}(\OM)
  \subset B^n(0,\sqrt{n}) \; \; \; \forall \nu.
\end{equation}
Finally, recalling the fact that each map $\Phi_{\nu}\circ L_{\nu}\circ \Lambda_{\nu}$ is a biholomorphism
and that, by construction, $\Phi_{\nu}\circ L_{\nu}\circ \Lambda_{\nu}(w_{\nu})=0$, we deduce from
\eqref{E:incl2} and from the fact that $s_{D} = s_{\OM}\big(\Psi(\bcdot)\big)$, the following:
\[
  s_D(z_{\nu})\,=\,s_{\OM}(w_{\nu})\,\geq\,\frac{\delta_0}{(2+\delta_0)n} \; \; \; \forall \nu.
\]
However, this contradicts the fact that $\lim_{\nu\to \infty}s_{D}(z_{\nu})=0$. Hence, our assumption above must be
false. Thus, $D$ is holomorphic homogeneous regular.
\end{proof}

\section{The proofs of Theorems~\ref{T:squee_0_I} and~\ref{T:squee_0_II}}\label{S:decay}

Before we prove the above-mentioned theorems, we require some definitions.

\begin{definition}\label{D:unramif_dom}
Let $D$ be a domain in $\C^n$ and let $\coll$ be a subset of $\hol(D)$.
\begin{enumerate}[leftmargin=27pt, label=$(\alph*)$]
  \item\label{I:manif} Let $X$ be a connected, Hausdorff topological space, and let $p: X\to \C^n$ be a local homeomorphism.
  Then, the pair $(X, p)$ is called an \emph{unramified domain} (over $\C^n$).
  
  \item Let $(X, p)$ be an unramified domain over $\C^n$. We call $(X, p, D, \inc)$
  an \emph{$\coll$-extension of $D$} if $\inc: D\to X$ is a continuous map such that
  $p\circ \inc = {\sf id}_D$ and such that for each $f\in \coll$ there exists a function $F_f\in \hol(X)$
  satisfying $F_f\circ\inc = f$. We call the function $F_f$ an \emph{analytic continuation} of $f$ to $X$.
  
  \item\label{I:env} An $\coll$-extension $(\wt{X}, \wt{p}, D, \wt{\inc})$ of $D$ is called an
  \emph{$\coll$-envelope of holomorphy}
  if for any $\coll$-extension $(X, p, D, \inc)$ of $D$, there exists a holomorphic map $h: X\to \wt{X}$ such that
  $\wt{p}\circ h = p$, $h\circ \inc = \wt{\inc}$ and\,---\,denoting by $\wt{F}_f$ and $F_f$ the analytic continuations
  of $f\in \coll$ to $\wt{X}$ and $X$, respectively\,---\,$\wt{F}_f\circ h = F_f$ for each $f\in \coll$.
\end{enumerate}  
\end{definition}

The following remark is, perhaps, in order.

\begin{remark}\label{R:in_order}
It is manifest from Definition~\ref{D:unramif_dom}-\ref{I:manif} that an unramified domain $(X,p)$ over
$\C^n$ acquires the structure of a complex manifold through the map $p$. This is presumed in the 
definition above of an $\coll$-extension. However, more can be said. With $D$ and $\coll$ as in
Definition~\ref{D:unramif_dom}:
\begin{itemize}[leftmargin=27pt]
  \item It follows easily from Definition~\ref{D:unramif_dom}-\ref{I:env} that the 
  data $(D, \coll)$ determines its $\coll$-envelope of holomorphy (assuming that an
  $\coll$-envelope of holomorphy exists) uniquely up to a
  biholomorphism.
  
  \item It is known that for any $\coll\subseteq \hol(D)$, the $\coll$-envelope of holomorphy exists.
  This is a well-known result of Thullen.
\end{itemize}
\end{remark}

We shall use the same notation as introduced in Sections~\ref{S:examples} and~\ref{S:scale}.
We are now in a position to prove the theorems to which this section is dedicated.

\begin{proof}[The proof of Theorem~\ref{T:squee_0_I}]
Fix a point $z\in D$. Let $F=(f_1,\dots, f_n): D\to B^n(0,1)$ be a holomorphic embedding such that
$F(z)=0$ and
$B^n(0, s_D(z))\subset F(D)$. The reason for the existence of such a map is discussed at the
beginning of the proof of Theorem~\ref{T:conv_to_1}.
Owing to the Hartogs phenomenon, 
$\exists \wt{f}_j\in \hol(\OM)$ such that
\[
  \left.\wt{f}_j\right|_D\,=\,f_j, \; \; j=1,\dots,n.
\]
Write $\wt{F}:=(\wt{f}_1,\dots, \wt{f}_n)$.
\medskip

\noindent{\textbf{Claim.} $\wt{F}(K\cap \bdy{D})\cap F(D) = \varnothing$.}
\vspace{0.65mm}

\noindent{Let us assume that $\wt{F}(K\cap \bdy{D})\cap F(D)\neq \varnothing$. Then, 
there exists a point $p_1\in K\cap \bdy{D}$ and a point $p_2\in D$ such that
$\wt{F}(p_1) = F(p_2) = w_0$. Now, consider a sequence $\{z_{\nu}\}\subset D$ such that
$z_{\nu}\to p_1$. Then, as $\wt{F}$ is continuous at $p_1$,
\[
  \lim_{\nu\to\infty}F(z_{\nu})\,=\,\lim_{\nu\to\infty}\wt{F}(z_{\nu})\,=\,w_0.
\]
As $w_0 = F(p_2)$, obviously $w_0\in F(D)$. Then, the above limit contradicts the
the fact that $F$ is a proper map onto $F(D)$.
So, our assumption must be false, which establishes the above claim.}%
\medskip

Define $u: \OM\to \R$ by
\[
  u(z)\,:=\,\sum_{j=1}^n|\wt{f}_j(z)|^2 \; \; \forall z\in \OM,
\]
which is clearly plurisubharmonic. As $u$ is continuous, there exists a point $p_3\in K$ such that
$u(p_3) = \sup_{K}u$. If $u(p_3)\geq 1$, then $p_3$ would be a point of global maximum of $u$
since $u(z) = \|F(z)\| < 1$ for each $z\in D = \OM\setminus K$, which would contradict the
plurisubharmonicity of the non-constant function $u$. Thus,
\begin{align} 
  {\sf range}(\wt{F})&\subseteq B^n(0,1), \label{E:in-ball} \\
  \wt{F}(K\cap \bdy{D})&\subset B^n(0,1). &\label{E:inside}.
\end{align}
By definition of $\wt{F}$ and our Claim above,
$\wt{F}(K\cap \bdy{D})\subset \bdy{F(D)}$.
Therefore
\begin{equation}\label{E:s_D_vs_Euc}
  s_D(z)\,=\,{\rm dist}\big(0, \bdy{F(D)}\big)\,\leq\,{\rm dist}\big(0, \wt{F}(K\cap \bdy{D})\big).
\end{equation}
From \eqref{E:inside}, \eqref{E:s_D_vs_Euc} and the fact that
$K_{B^n(0,1)}(0, \bcdot) = \tanh^{-1}\|\bcdot\|$, we have
\begin{equation}\label{E:s_D_vs_K}
  s_{D}(z)\,\leq\,\tanh\big(K_{B^n(0,1)}\big(0,\wt{F}(K\cap \bdy{D})\big)\big).
\end{equation}
By \eqref{E:in-ball}, $\wt{F}$ is $B^n(0,1)$-valued map. By the distance-decreasing property:
\[
  K_{B^n(0,1)}\big(0,\wt{F}(K\cap \bdy{D})\big)\,\leq\,K_{\OM}(z, K\cap \bdy{D}).
\]
Combining the latter inequality with \eqref{E:s_D_vs_K} gives \eqref{E:squee_0}.  
\end{proof}

\begin{proof}[The proof of Theorem~\ref{T:squee_0_II}]
Fix a point $z\in D$. Let $(\enve{D}, \pi, D, \inc)$ denote the data describing
the $\coll$-envelope of holomorphy of $D$\,---\,see Definition~\ref{D:unramif_dom}. 
Let $F=(f_1,\dots, f_n): D\to B^n(0,1)$ be a holomorphic embedding such that $F(z)=0$
and $B^n(0, s_D(z))\subset F(D)$ (which exists; see discussion above and
\cite[Theorem~2.1]{dengGuanZhang:spsfbd12}). There exists 
$\wt{f}_k\in \hol(\enve{D})$ such that
\[
  \wt{f}_k\circ\inc\,=\,f_k, \; \; k=1,\dots, n.
\]
Write $\wt{F}:=(\wt{f}_1,\dots, \wt{f}_n)$, and set
\[
  A\,:=\,\bdy_{\coll}\big(\inc(D)\big)\,\neq\,\varnothing,
\]
where $\bdy_{\coll}E$ denotes the topological boundary of any set $E\subset \enve{D}$.
That $A\neq \varnothing$ follows from our hypothesis that $\iness{D}\neq \varnothing$: in
fact, a consequence of Thullen's construction\,---\,referenced in Remark~\ref{R:in_order}\,---\,is
that
$\pi(A)\,=\,\iness{D}$.
Given the complex structure that $\enve{D}$ is endowed with, $\inc$ is a holomorphic embedding.
Thus, as
\[
  \wt{F}\circ\inc\,=\,F \quad\text{and}
  \quad\text{$\inc: D\to \inc(D)$ is proper onto $\inc(D)$},
\]
it follows from an arguement analogous to that of the Claim in the previous proof that
\begin{equation}\label{E:sept}
 \wt{F}(A)\cap F(D)\,=\,\varnothing.
\end{equation}
Now, \textbf{fix} a constant $\eps:\,0<\eps\ll 1$ and let $\OM\subset \enve{D}$ be a connected open set
such that
\[
  A\subset \OM \quad\text{and}
  \quad \wt{F}(\OM)\subset B^n(0, 1+\eps).
\]
Let us write $\Phi_{\eps}:=\left. (1+\eps)^{-1}\wt{F}\right|_{\OM}$. By construction:
\begin{align} 
  {\sf range}(\Phi_{\eps})&\subseteq B^n(0,1), \label{E:in-ball_II} \\
  \Phi_{\eps}(A)&\subset B^n(0,1). &\label{E:inside_II}.
\end{align}
As $\wt{F}\circ\inc = F$, by the definition of $A$ and \eqref{E:sept}, we have
$\Phi_{\eps}(A)\subset \bdy\Phi_{\eps}\circ\inc(D)$.
Thus
\begin{equation}\label{E:s_D_vs_Euc_II}
  s_D(z)\,=\,{\rm dist}\big(0, \bdy{F(D)}\big)\,=\,(1+\eps)\,{\rm dist}
  \big(0, \bdy{\Phi_{\eps}\circ\inc(D)}\big)\,\leq\,(1+\eps)\,{\rm dist}\big(0, \Phi_{\eps}(A)\big).
\end{equation}
From \eqref{E:inside_II} and \eqref{E:s_D_vs_Euc_II}, we have (arguing as in the previous proof)
\[
  (1+\eps)^{-1}s_{D}(z)\,\leq\,\tanh\big(K_{B^n(0,1)}\big(0,\Phi_{\eps}(A)\big)\big).
\]
By \eqref{E:in-ball_II}, $\Phi_{\eps}$ is $B^n(0,1)$-valued map. Thus, by the distance-decreasing property
(note that $\Phi_{\eps}\circ\inc(z) = 0$) and the last inequality: 
\begin{equation}\label{E:penult}
  s_{D}(z)\,\leq\,(1+\eps)\,\tanh\big(K_{B^n(0,1)}\big(0,\Phi_{\eps}(A)\big)\big)\,\leq\,(1+\eps)\,
  K_{\OM}(\inc(z), A).
\end{equation}

Let us fix a point $p\in \iness{D}$. Let $\{z_{\nu}\}\subset D$ be some sequence such that 
$z_{\nu}\to p$. We would have the desired result if we could show that for any arbitrary
subsequence $\{z_{\nu_k}\}$ there exists a subsequence 
$\big\{z_{\nu_{k_l}}\big\}\subset \{z_{\nu_k}\}$ such that
\begin{equation}\label{E:goal}
  \lim_{l\to \infty}s_{D}\big(z_{\nu_{k_l}}\big)\,=\,0.
\end{equation}
To this end, fix $\{z_{\nu_k}\}$ and consider the sequence $\{\inc(z_{\nu_k})\}\subset \enve{D}$.
Since $\inc: D\to \inc(D)$ is proper onto $\inc(D)$, there exists a point $\xi\in A$ and
a subsequence $\big\{z_{\nu_{k_l}}\big\}$ such that
\[
  \inc\big(z_{\nu_{k_l}}\big)\to \xi \; \; \text{as $l\to \infty$}.
\]
From this and \eqref{E:penult}, the conclusion \eqref{E:goal}\,---\,and therefore the result\,---\,follows.
\end{proof}

\section{A few questions on the squeezing function}\label{S:ques}

We conclude this section with a few questions that are suggested by either the discussion in Section~\ref{S:intro}
or by the proofs in the previous sections.
\smallskip

Before presenting the first question, we provide some context for it. Let $D$ be a bounded domain in $\C^n$,
$n\geq 2$, and let $p\in \bdy{D}$. Suppose $\bdy{D}$ is $\smoo^\infty$-smooth and Levi-pseudoconvex near $p$,
and let $\{z_\nu\}$ be a sequence in $D$ with paraboloidal approach to $p$. If the Levi-form of $\bdy{D}$
at $p$ has at least $(n-2)$ positive eigenvalues and $\bdy{D}$ is assumed to be of finite type
near $p$ then higher-dimensional analogues  of all the statements that make Lemma~\ref{L:conv_to_1_key}
possible hold true\,---\,see \cite[Section~1]{bedfordPinchuk:dCn+1nag91}. This, along with
Remark~\ref{R:BS_discuss} (see \ref{I:bihol} below for the connection),
motivates the following:

\begin{question}\label{Q:corank_1}
Let $D$ be a bounded domain in $\C^n$, $n\geq 2$, let $p\in \bdy{D}$, suppose $\bdy{D}$ is
$\smoo^\infty$-smooth near $p$ and that the Levi-form of $\bdy{D}$ at $p$ has 
at least $(n-2)$ positive eigenvalues. Assume $\bdy{D}$ is Levi-pseudoconvex and of finite
type near $p$. If, for some sequence $\{z_\nu\}\subset D$
with paraboloidal approach to $p$, $\lim_{\nu\to \infty}s_D(z_{\nu}) = 1$, then does it follow that
$p$ is a strongly pseudoconvex point?%
\end{question}

The discussion in Remark~\ref{R:BS_discuss} is the motivation for considering $\{z_{\nu}\}$
having paraboloidal approach (see its relevance in \ref{I:bihol} below).
Other factors that motivate Question~\ref{Q:corank_1}
involve the other key inputs to proving Theorem~\ref{T:conv_to_1}. Specifically:
%\begin{itemize}[leftmargin=24pt]
\begin{enumerate}[leftmargin=27pt, label=$(\roman*)$]
  \item\label{I:conv} Let $\ome$ be a domain in $\C^n$, \textbf{not necessarily} bounded, and
  let $\xi\in \bdy{\ome}$
  be such that $\ome$ has reasonable boundary-geometry around $\xi$. 
  One needs a characterisation, given some fixed domain $\OM$ in $\C^n$, for a sequence
  $\{\varPsi_{\nu}\}\subset {\rm Hol}(\OM, \ome)$ to converge uniformly on compact subsets to $\xi$.
  Result~\ref{R:normal_conv} is a result of this sort whose hypothesis is relevant to the domains
  that arise in an argument that mimics the proof of Theorem~\ref{T:conv_to_1} with $D$ as in
  Question~\ref{Q:corank_1}.
  
  \item\label{I:bihol} While a version of Result~\ref{R:oelj} in $\C^n$, for general $n\geq 2$, with a hypothesis
  that is the obvious generalisation of the hypothesis of Result~\ref{R:oelj}, is unknown, it is 
  reasonable to ask whether the \emph{weaker} conclusion in the last sentence of
  Result~\ref{R:oelj} is true if, say, $P$ and $Q$ are quadratic in $(n-2)$ variables. In fact,
  the latter question may be of independent interest. Its relevance to Question~\ref{Q:corank_1}
  is suggested by the last two sentences of Remark~\ref{R:BS_discuss}.
\end{enumerate}  
The above provides additional context to Question~\ref{Q:corank_1} (and hints at an affirmative answer). 
\smallskip

Our next question is an analogue of Question~\ref{Q:corank_1} for convex domains. Let $D$ be a bounded
convex domain in $\C^n$,
$n\geq 2$, and let $p\in \bdy{D}$. Suppose $\bdy{D}$ is $\smoo^\infty$-smooth and of finite type near $p$.
This time, if $\{z_\nu\}$ is a sequence in $D$ with paraboloidal approach to $p$, then, while it is unclear
that the scaling method described in Section~\ref{S:scale} can be used to deduce a higher-dimensional
analogue of Lemma~\ref{L:conv_to_1_key}, convexity enables one to use a simpler scaling method to
achieve this. This is a method of affine scaling, of which there are many accounts in the literature.
An account that is the most relevant in the present case is the one in \cite[Section~2]{gaussier:ccdnag97}.
The hypothesis of finiteness of type of $\bdy{D}$ near $p$ also implies that an analogue
of our observation \ref{I:conv} above holds true in the present setting. These points suggest the
following:

\begin{question}\label{Q:convex}
Let $D$ be a bounded convex domain in $\C^n$, $n\geq 2$, let $p\in \bdy{D}$, suppose $\bdy{D}$ is
$\smoo^\infty$-smooth and of finite type near $p$. If, for some sequence $\{z_\nu\}\subset D$
with paraboloidal approach to $p$, $\lim_{\nu\to \infty}s_D(z_{\nu}) = 1$, then does it follow that
$p$ is a strongly pseudoconvex point?%
\end{question} 

We must mention that, given a bounded convex domain $D\subset \C^n$, $n\geq 2$, the
issue of whether $\lim_{z\to p}s_{D}(z)=1$ implies that $p$ is a strongly pseudoconvex point
is understood for $\bdy{D}$ having low regularity\,---\,see \cite{zimmer:cspobLe19}
by Zimmer. However, Zimmer's approach requires, essentially, that $\lim_{z\to p}s_{D}(z)=1$
for \textbf{all} $p\in \bdy{D}$ to get the stronger conclusion that $D$ is strongly pseudoconvex.
This, therefore, is a phenomenon different from the one that Question~\ref{Q:convex} relates
to.
\smallskip

Our last question arises immediately from the discussion above and the discussion that precedes
Theorem~\ref{T:conv_to_1}. Loosely speaking, one asks: given a domain $D\Subset \C^n$,
$n\geq 2$, a point $p\in \bdy{D}$, and given that $\bdy{D}$ has very regular Levi geometry
around $p$, how predictive is the existence of a sequence $\{z_{\nu}\}\subset D$ approaching
$p$, and satisfying $\lim_{\nu\to \infty}s_{D}(z_{\nu})=1$, of strong pseudoconvexity of
$\bdy{D}$ at $p$? This question gains further significance in view Remark~\ref{R:BS_discuss}.
A specific formulation of it is as follows:

\begin{question}\label{Q:predictive}
Does there exist a domain $D\Subset \C^n$, $n\geq 2$, that admits
\begin{itemize}[leftmargin=24pt]
  \item a point $p\in \bdy{D}$ such that the pair $(D,p)$ satisfies the conditions stated in
  Theorem~\ref{T:conv_to_1} or Question~\ref{Q:corank_1} or Question~\ref{Q:convex}; and
  \item some sequence $\{z_{\nu}\}\subset D$ approaching $p$ such that
  $\lim_{\nu\to \infty}s_{D}(z_{\nu})=1$;
\end{itemize}
and such that $\bdy{D}$ is not strongly Levi pseudoconvex?
\end{question}
\medskip

\section*{Acknowledgments}
\noindent{This work is supported in part by a UGC CAS-II grant (No.~F.510/25/CAS-II/2018(SAP-I)).
The author thanks Sushil Gorai for a series of interesting talks on the squeezing function, given at the
2019 Discussion Meeting on Several Complex Variables, which
introduced the author to the subject of this paper. He also thanks the Kerala School of Mathematics, Kozhikode, India, for
hosting the 2019 meeting.}

\end{document}